\documentclass[10pt]{amsart}
\usepackage{amsmath,amsfonts,latexsym,graphicx,amssymb,url}
\usepackage{amsmath,txfonts,pifont,bbding,pxfonts,manfnt}
\usepackage{psfrag}
\usepackage{wrapfig, subfigure,graphicx}
\usepackage{hyperref}
\usepackage[active]{srcltx}
\usepackage{pdfsync}
\usepackage{amssymb}
\usepackage{amsthm}
\usepackage{amsmath}
\usepackage{graphicx}
\usepackage{color}
\usepackage{mathtools}

\allowdisplaybreaks
\numberwithin{equation}{section}
\newcommand{\R}{{\mathbb R}} 

\theoremstyle{plain}
\newtheorem{theorem}{Theorem}[section]

\newtheorem{remark}[theorem]{Remark}

\providecommand{\bysame}{\makebox[3em]{\hrulefill}\thinspace}

\begin{document}

\setcounter{equation}{0}










\title[Aspherical lens design and Imaging]
{Aspherical lens design and Imaging}
\author[C. E. Guti\'errez and A. Sabra]
{Cristian E. Guti\'errez\\
 and \\
Ahmad Sabra}
\thanks{\today. Authors partially supported by NSF grant DMS-1201401.
The results in this paper were announced in the Freeform Optics Meeting, Arlington, VA, June 2015, see\cite{arlingtonannouncement:gutierrez-sabra}.}
\thanks{AMS Subject Classification 35F50, 35Q60, 78A05, 78A46. Key words: lens design, pdes, inverse problems, geometric optics.}
\address{Department of Mathematics\\Temple University\\Philadelphia, PA 19122}
\email{gutierre@temple.edu}
\address{Current address of A. Sabra: Faculty of Mathematics, Informatics, and Mechanics,
University of Warsaw, Poland}
\email{sabra@mimuw.edu.pl}
\begin{abstract}
We design freeform lenses refracting an arbitrarily given incident field into a given fixed direction. In the near field case, we study the existence of lenses refracting a given bright object into a predefined image. We also analyze the existence of systems of mirrors that solve the near field and the far field problems for reflection.
\end{abstract}

\maketitle


\setcounter{equation}{0}

\section{Introduction} 


The goal of this paper is to show the existence of a lens that refracts a given variable unit field of directions, emitted from a planar source, into a predefined constant direction when the face closer to the source is given. 
We assume the lens is made of a material denser than the exterior medium. In order to avoid total reflection, we require the lower face of the lens, the incident field, and the predefined direction to satisfy the compatibility condition described in \eqref{eq:conditionforrefraction}. We use Snell's law of refraction in vector form, which after analysis, reduces the problem to solving a system of first order partial differential equations. Assuming \eqref{eq:conditionforrefraction}, we prove that a solution exists if and only if the incident field satisfies a curl condition, see \eqref{eq:curlzerocondition}. To minimize internal reflection, the first face of the lens can be chosen to be orthogonal to the incident field; this is described in Example \ref{sec:avoidreflection}.


Using the far field analysis, we solve the following imaging problem in the near field case. Assume $T$ is a bijective transformation between the planar source $\Omega$ and a plane domain $\Omega^*$. The goal is to find a lens -both faces to be determined- that refracts collimated radiation emitted from $\Omega$ into $\Omega^*$. The lower face of the lens satisfies a system of nonlinear partial differential equations \eqref{eq:secondformpde}. We find conditions on $T$ so that a solution to the corresponding PDE exists. Once the lower face is found, the second face can be constructed using the results of Section \ref{sec:verticalcase}. We use this analysis to design lenses when the image is a magnification/contraction of the planar source, see Sections \ref{sec:Magnification} and \ref{sec:legendretransform}. 
Adapting the methods from Sections \ref{sec:farfield} and \ref{sec:nearfield}, we solve in Section \ref{sec:reflectioncase} the far field and the near field problems for reflection. In this case, the lens is replaced by a system of two mirrors.

We remark that the refractive/reflective surfaces considered are not necessarily rotationally symmetric, and the solutions found are explicit. This permits us to calculate and simulate numerically examples of application.

To place our results in perspective, both from the theoretical and practical points of view, we mention the following. 
The problem of finding a convex, analytic, and symmetric lens focusing all rays from a point source into a point image was first solved in \cite{friedmanmacleod:optimaldesignopticalens} in 2d using a fixed point type argument. 
This result is generalized in \cite{gutierrez:asphericallensdesign} to 3d to construct freeform $C^2$ lenses that refract rays emitted from a point source into a constant direction or a point image; the reflection case is studied in \cite{gutierrez-sabra:designofpairsofreflectors}. 
Illumination problems with one point source are studied in \cite{gutierrez-huang:nearfieldrefractor}, \cite{gutierrez-mawi:refractorwithlossofenergy} for the refraction case and in
\cite{Caf-K-Oli:antenna}, \cite{Wang:antenna}, \cite{oliker-kochengin:nearfieldantenna}, \cite{gutierrez-sabra:thereflectorproblemandtheinversesquarelaw} for reflection. The case when the input rays are collimated is considered in 
 \cite{gutierrez-tournier:parallelrefractor}, \cite{gutierrez-tournier:nearfieldrefraction}.

The use and design of free form surfaces in modern optics received great attention recently due to applications to illumination, imaging, aerospace and biomedical engineering.
For applications to illumination problems see \cite{Ding:FreeFormLED} and \cite{Oliker:SupportingOpticalFreeFormIllumination}, for aerospace engineering see \cite{Saunder:MetrologyFreeFormSCUBA}; and for automobile applications such as head-up displays and car driver-side mirrors design see \cite{Ott:HUDfreeform} and \cite{hick-perline:drivermirror}, respectively.
In addition, see \cite{Fang:ManufatAndMeasFreeForm} for an extended number of related references.
Due to recent technological advancements in ultra precision cutting, grinding, and polishing machines,
manufacturing free form optical devices with high precision is now possible, see \cite{Bialock:FabricFreeForm}. The systems obtained enhance the performance of traditional designs and provide more flexibility for designers \cite{Duerr:BenefitsFreeFormOptics}. 
Moreover, they can achieve imaging tasks that are impossible with the symmetric designs. For example, 
in \cite{Hick:CoupledFreeFormSurface} and \cite{Croke:FourmirrosFreeform}, the authors construct multiple free form mirror systems that increase the field of view of an observer and improve visual performance.  

The free form surfaces designed in this paper are a one parameter family given parametrically and therefore they may have self intersections. To obtain physically realizable surfaces when the input and output fields are collimated see Remark \ref{rmk:self-interesection}; the general case will be analyzed in the forthcoming paper \cite{gutierrez-sabra:farfield-with-two-surfaces}. 
In the present paper we do not consider input and output radiation intensities which requires other mathematical methods. For example, when $\sigma_1$ is optically inactive, this leads to solving Monge-Amp\`ere type equations, see for example \cite[Sect. 7.7]{winston-minano-benitez:nonimagingoptics}, \cite{ries-muschaweck:tailoredopticalsurfaces}, and \cite{Castro2015}. However, the results obtained here will be used in \cite{gutierrez-sabra:farfield-with-two-surfaces} to design lenses, with both faces optically active, and where the input and output energies are taken into consideration.


The paper is organized as follows. In Section \ref{sec:snelllaw}, we briefly introduce the Snell's law of refraction in vector form. In Section \ref{sec:farfield}, we precisely state and solve the far field problem. 
We introduce  in Section \ref{sec:nearfield} the near field imaging problem and formulate the corresponding system of PDEs \eqref{eq:secondformpde}. 
We find, in Section \ref{sec:sameindex}, necessary and sufficient conditions for the imaging map $T$  to solve the system of PDEs when both the source and the target are in media with same refractive index, Theorem \ref{thm:resultwhensameindex}. With the explicit formulas obtained, we next study in Section \ref{sec:Magnification} the case when the map is a magnification and analyze the thickness of the lens obtained. In Sections \ref{sec:n3smaller} and \ref{sec:n3larger}, we solve the imaging problem when the target and the source are in media with different refractive indices. In this case, we find sufficient conditions for existence of local solutions, see Theorem \ref{eq:existensolutionpde} and Section \ref{sec:n3larger}. The same problem is solved in Section \ref{sec:legendretransform}  in the plane with a different method using the Legendre transform as in \cite{gutierrez:refractionlegtransform} and when the map is a magnification. Finally, in Section \ref{sec:reflectioncase}, we solve the far field and the imaging problem for reflectors.



\setcounter{equation}{0}

\section{Snell's law of refraction in vector form}\label{sec:snelllaw}
Suppose $\Gamma$ is a surface in $\R^3$ that separates two media
I and II that are homogeneous and isotropic. Let $v_1$ and
$v_2$ be the velocities of propagation of light in the media I and
II respectively. The index of refraction of the medium I is by
definition $n_1=c/v_1$, where $c$ is the velocity of propagation
of light in vacuum, and similarly the refractive index of the medium II is $n_2=c/v_2$. If a
ray of light\footnote{Since the refraction angle depends on the frequency of the radiation, we assume light rays are monochromatic.} having direction $x\in S^{2}$, the unit sphere in $\R^{3}$, and traveling
through the medium I strikes $\Gamma$ at the point $P$, then this ray
is refracted into the direction $m\in S^{2}$ through the medium II
according to the Snell law in vector form:
\begin{equation}\label{eq:snellwithcrossproduct}
n_{1}(x\times \nu)=n_{2}(m\times \nu),
\end{equation} 
where $\nu$ is the unit normal to the surface to $\Gamma$ at $P$ pointing towards medium II; see \cite[Subsection 4.1]{luneburg:maththeoryofoptics}. It is tacitly assumed here that $x\cdot \nu \geq 0$.
We deduce from \eqref{eq:snellwithcrossproduct} the following
\begin{enumerate}
\item[(a)] the vectors $x,m,\nu$ are all on the same plane (called plane of incidence);
\item[(b)] the well known Snell's law in scalar form
$$n_1\sin \theta_1= n_2\sin
\theta_2,$$ 
where $\theta_1$ is the angle between $x$ and $\nu$
(the angle of incidence),
$\theta_2$ the angle between $m$ and $\nu$ (the angle of refraction).
\end{enumerate}
From \eqref{eq:snellwithcrossproduct}, $(n_{1}x-n_{2}m)\times \nu=0$, which means that the
vector $n_{1}x-n_{2}m$ is parallel to the normal vector $\nu$.
If we set $\kappa=n_2/n_1$, then
\begin{equation}\label{eq:snellvectorform}
x-\kappa \,m =\lambda \nu,
\end{equation}
for some $\lambda\in \R$. Notice that \eqref{eq:snellvectorform} determines $\lambda$. In fact, taking dot product with $\nu$ we get
$\lambda=\cos \theta_1-\kappa \cos \theta_2$,
$\cos \theta_1=x\cdot \nu\geq 0$, and
$\cos \theta_2=m\cdot \nu=\sqrt{1-\kappa^{-2}[1-(x\cdot \nu)^2]}$,
so
\begin{equation}\label{eq:formulaforlambda}
\lambda=x\cdot \nu -\kappa \,\sqrt{1-\kappa^{-2}\left(1-(x\cdot \nu)^{2}\right)}.
\end{equation}
Refraction behaves differently for $\kappa<1$ and for $\kappa>1$.
Indeed, by \cite[Subsection 2.1]{gutierrez-huang:farfieldrefractor}
\begin{enumerate}
\item[(C1)]\label{item:conditionkappa<1} if $\kappa<1$ then for refraction to occur we need $x\cdot \nu \geq \sqrt{1-\kappa^2}$, and in \eqref{eq:formulaforlambda} we have $\lambda>0$, and the refracted vector $m$ is so that $x\cdot m\geq \kappa$;
\item[(C2)]\label{item:conditionkappa>1} if $\kappa>1$ then refraction always occurs and we have $x\cdot m\geq 1/\kappa$, and in \eqref{eq:formulaforlambda} we have $\lambda<0$.
\end{enumerate}


\setcounter{equation}{0}
\section{Refraction of a general field $e(x)$}\label{sec:farfield}

We begin stating the problem to be solved.
We are given a $C^2$ surface $\sigma_1$ in $\R^3$ (the bottom face of the lens), and a differentiable unit field $e(x)=(e_1(x),e_2(x),e_3(x)):=(e'(x),e_3(x))$, with $e_3(x)>0$, that is defined for every $x=(x_1,x_2)$ in a plane domain $\Omega$. 
We want to construct a second surface $\sigma_2$ (the top face of the lens) so that each ray emitted from $(x,0)$ with direction $e(x)$ is refracted by the lens sandwiched by $\sigma_1$ and $\sigma_2$ into a given fixed unit direction $w$. We assume that the medium below $\sigma_1$ has refractive index $n_1$, the medium above $\sigma_2$ has refractive index $n_3$, and the material enclosed within the lens has refractive $n_2$, such that $n_2>n_1, n_3$, but $n_1$ and $n_3$ are unrelated; see Figure \ref{fig:picoflensgeneralfield}. 

%

Define $\kappa_1=n_2/n_1, \kappa_2=n_3/n_2$, and assume that the lower surface $\sigma_1$ of the lens is given by the graph of a $C^2$ function $u$. A ray emanating from each point $(x,0)$, $x\in \Omega$, with unit direction $e(x)$ strikes $\sigma_1$ at some point $P(x)=(\varphi(x),u(\varphi(x)))$; we assume that the map $x\mapsto \varphi(x)$ is $C^1$.

At $P(x)$, let $\nu_1(x)$ be the unit normal to $\sigma_1$ pointing towards medium $n_2$. We will assume that $e(x)\cdot \nu_1(x)\geq 0$.  Since $\kappa_1>1$, by $(C2)$ the ray is refracted at $P(x)$ into the lens with unit direction $m_1(x)$ given from the vectorial Snell's law \eqref{eq:snellvectorform} by
\begin{equation}\label{eq:formulaform1}
e(x)-\kappa_1 m_1(x)=\lambda_1 \nu_1(x),
\end{equation}
where $\lambda_1(x)=e(x)\cdot \nu_1(x) -\kappa_1\sqrt{1-\kappa_1^{-2}\left(1-(e(x)\cdot \nu_1(x)\right)^2)}<0$. 
The ray then continues traveling inside the lens and strikes the yet unknown surface $\sigma_2$ at some point $Q(x)$. 
At $Q(x)$, the incident ray has direction $m_1(x)$ and we want it to be refracted into the direction $w$ in medium $n_3$.

\begin{figure}
\includegraphics[width=3in]{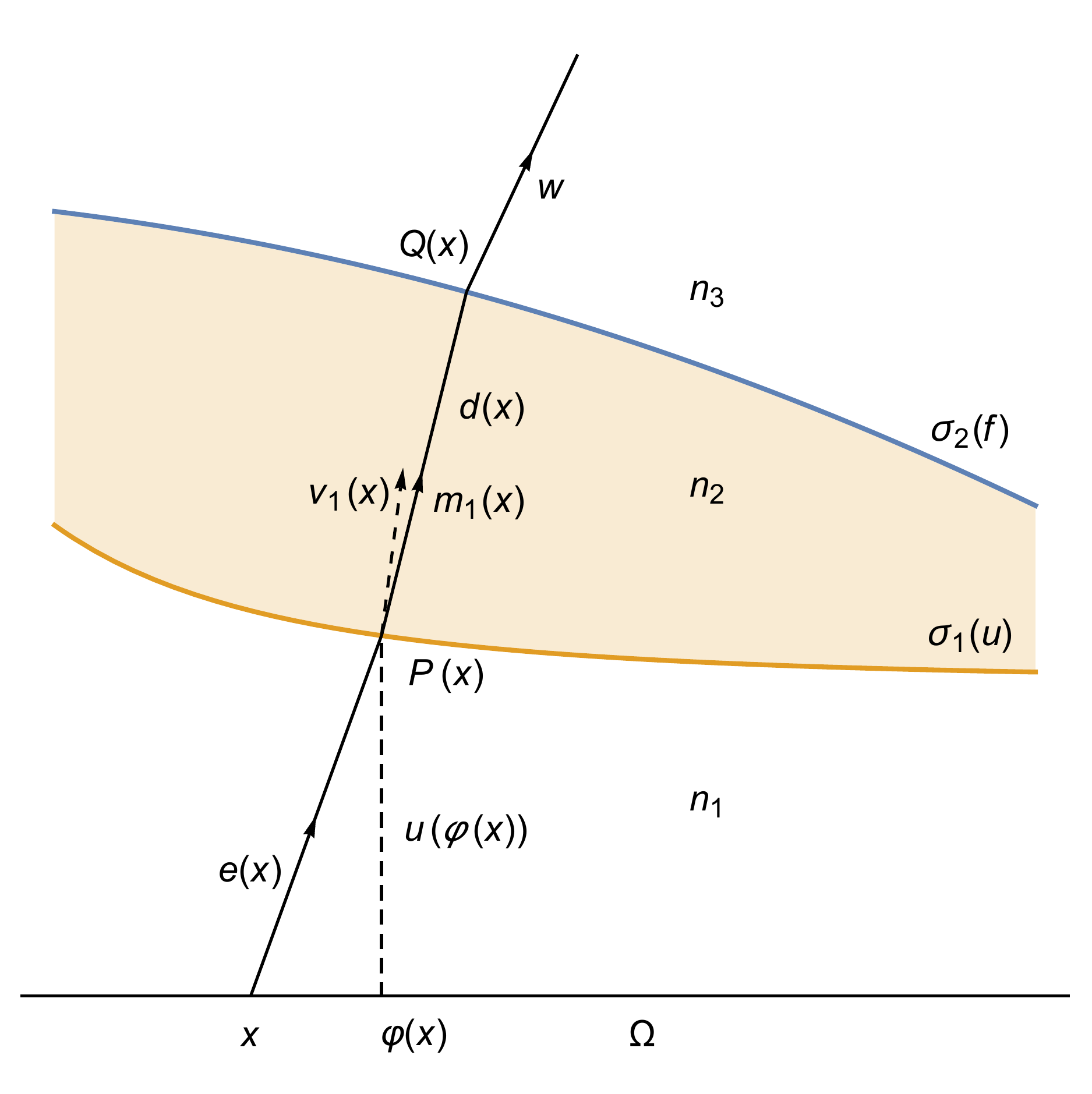}
\caption{Far field model}
\label{fig:picoflensgeneralfield}
\end{figure}

We introduce the distance 
$$d(x)=|P(x)-Q(x)|,$$
and parametrize the unknown surface $\sigma_2$ by the vector 
\begin{equation}\label{eq:formulafortherightcurveefield}
f(x):=(\varphi(x),u(\varphi(x)))+ d(x)\,m_1(x),\quad x=(x_1,x_2)\in \Omega,\qquad \varphi(x)=(\varphi_1(x),\varphi_2(x)).
\end{equation}

The purpose of this section is to construct a family of surfaces $\sigma_2$ by proving the following existence theorem.
\begin{theorem}\label{thm:existenceofupperlenssurfaces}
We are given a $C^2$ surface $\sigma_1$ parametrized by $(\varphi(x),u(\varphi(x)))$, with $x\mapsto \varphi(x)$ $C^1$,  
a $C^1$ unit field $e(x)=(e'(x),e_3(x))$, and a unit direction $w$.
A surface $\sigma_2\in C^2$ such that the lens $(\sigma_1,\sigma_2)$ refracts all rays with direction $e(x)$, $x\in \Omega$, into $w$ exists if and only if $m_1(x)\cdot w\geq \kappa_2$ for all $x\in \Omega$ and ${\rm curl} \,e'(x)=0$.
Moreover, if $\nabla h(x) =e'(x)$ for some $h$, $\sigma_2$ is then parametrized by 
$f(x)=(\varphi(x),u(\varphi(x)))+d(x,C,w)\,m_1(x)$ where $m_1$ is given by \eqref{eq:formula2form1} and $d$ by \eqref{eq:formulaford(x)forgeneralfield}.
$C$ is a constant chosen so that $d>0$.
\end{theorem}

\begin{proof}
By \eqref{eq:formulaform1}
\begin{equation}\label{eq:formula2form1}
m_1(x)=\dfrac{1}{\kappa_1}(e(x)-\lambda_1(x)\nu_1(x)).
\end{equation}
Since $\kappa_2<1$, from $(C1)$, for refraction to occur at $Q(x)$ and to avoid total reflection, we need $m_1\cdot w\geq \kappa_2$, which from \eqref{eq:formula2form1} is equivalent to
\begin{equation}\label{eq:conditionforrefraction}
\lambda_1(x)\nu_1(x)\cdot w \leq e(x)\cdot w-\kappa_1\kappa_2,
\end{equation}
for all $x\in \Omega$. This is a compatibility condition between $e$, $\sigma_1$, and $w$.

Now, applying the Snell's law at $Q(x)$ we have that
\begin{equation}\label{eq:refractionatsigma2}
 m_1-\kappa_2 w =\lambda_2\nu_2,
\end{equation}
where $\nu_2$ is the unit normal to $\sigma_2$ at $Q(x)$ pointing towards medium $n_3$.

From \eqref{eq:formulafortherightcurveefield}, to find $f(x)$ it only remains to calculate $d(x)$. We note,  from \eqref{eq:formula2form1} and \eqref{eq:refractionatsigma2}, that the vector $e(x)-\lambda_1\nu_1-\kappa_1\kappa_2w$ is a multiple of the normal $\nu_2$. Therefore this vector must be perpendicular to the tangent plane to $\sigma_2$, which is equivalent to the following system of first order partial differential equations
\begin{align}
\left(e(x)-\lambda_1\nu_1-\kappa_1\kappa_2w\right)\cdot f_{x_1}&=0 \label{eq1}\\
\left(e(x)-\lambda_1\nu_1-\kappa_1\kappa_2w\right)\cdot f_{x_2}&=0\label{eq2}.
\end{align}



To carry out the calculations, we will assume  that $d(x)$ is $C^1$. Later on, we will show that this is a consequence of the assumptions on $u$ and $e$.
In fact, by \eqref{eq:formula2form1}, we have that
\[
f_{x_1}(x)=\left(\varphi(x),u(\varphi(x))\right)_{x_1}+ \left[ d(x)\, \dfrac{1}{\kappa_1}(e(x)-\lambda_1 \,\nu_1) \right]_{x_1}.
\]
Hence from \eqref{eq1}
\begin{align*}
0&= \left(e(x)-\lambda_1\,\nu_1 -\kappa_1\kappa_2 \,w \right)\cdot \left[ \left(\varphi(x),u(\varphi(x))\right)_{x_1}+ \left[ d(x)\, \dfrac{1}{\kappa_1}(e(x)-\lambda_1 \,\nu_1) \right]_{x_1}\right]\\
&=
\left(e(x)-\lambda_1\,\nu_1 -\kappa_1\kappa_2 \,w\right])\cdot  \left(\varphi(x),u(\varphi(x))\right)_{x_1}
+ (e(x)-\lambda_1\,\nu_1)\cdot \left[ d(x) \dfrac{1}{\kappa_1}(e(x)-\lambda_1 \,\nu_1) \right]_{x_1}\\
&\qquad \qquad- \kappa_1\kappa_2 \,w\cdot \left[ d(x) \dfrac{1}{\kappa_1}(e(x)-\lambda_1 \,\nu_1) \right]_{x_1}\\
&= F(x) + A(x) -B(x).
\end{align*}
We calculate $A(x)$. We have that $(e(x)-\lambda_1\,\nu_1)\cdot (e(x)-\lambda_1\,\nu_1)=|e(x)-\lambda_1\,\nu_1|^2=\kappa_1^2|m_1|^2=\left(\kappa_1\right)^2$, and therefore
$(e(x)-\lambda_1\,\nu_1)\cdot (e(x)-\lambda_1\,\nu_1)_{x_1}=0$.
So 
$
A(x)=\kappa_1\,d_{x_1}(x)$.

Also, since $\nu_1(x)$ is the normal to $\sigma_1$ at the point $(\varphi(x),u(\varphi(x)))$, 
we have that $\nu_1\cdot \left(\varphi(x),u(\varphi(x))\right)_{x_1}=0$, and hence
$
F(x)= \left(e(x)-\kappa_1\kappa_2\,w\right)\cdot \left(\varphi(x),u(\varphi(x))\right)_{x_1}$.
We obtain that $d$ satisfies the following pde
\[
\kappa_1\,d_{x_1}- \kappa_2\,w\cdot \left[(e(x)-\lambda_1 \,\nu_1) d(x)\right]_{x_1}=
-\left(e(x)-\kappa_1\kappa_2 \,w\right)\cdot  \left(\varphi(x),u(\varphi(x))\right)_{x_1} 
\]
which, since $w$ is a constant vector, can be written as 
\begin{equation}\label{eq:odeforfarfield}
\left[\left( \kappa_1- \kappa_2\,w\cdot (e(x)-\lambda_1 \,\nu_1) \right)d(x)\right]_{x_1}
=
-\left(e(x)-\kappa_1\kappa_2 \,w\right)\cdot  \left(\varphi(x),u(\varphi(x))\right)_{x_1}.
\end{equation}
Analogously, using \eqref{eq2} we obtain that $d$ also satisfies the equation 
\begin{equation}\label{eq:odeforfarfieldv}
\left[\left( \kappa_1- \kappa_2\,w\cdot (e(x)-\lambda_1 \,\nu_1) \right)d(x)\right]_{x_2}
=
-\left(e(x)-\kappa_1\kappa_2 \,w\right)\cdot  \left(\varphi(x),u(\varphi(x))\right)_{x_2}.
\end{equation}
Define 
\begin{align*}
H(x)&=\left(\kappa_1- \kappa_2\,w\cdot (e(x)-\lambda_1 \,\nu_1) \right)d(x)\\
G(x)&= \left(e(x)-\kappa_1\kappa_2 \,w\right)\cdot  \left(\varphi(x),u(\varphi(x))\right).
\end{align*}
Notice that since $u\in C^2$, then $\nu_1\in C^1$, and since $e,\varphi \in C^1$ and we assume $d\in C^1$, we have that $H$ and $G$ are both $C^1$.
We have that
\begin{align*}
\left(e(x)-\kappa_1\kappa_2 \,w\right)\cdot  (\varphi(x),u(\varphi(x)))_{x_i}
=G_{x_i}(x) - e_{x_i}(x)\cdot \left(\varphi(x),u(\varphi(x))\right)
\end{align*}
The ray with direction $e(x)$ strikes the surface $\sigma_1$ at the point  $(\varphi (x),u(\varphi(x)))$, then
\[
\left(\varphi(x),u(\varphi(x))\right)=\rho(x) \,e(x) +(x,0)
\]
for some positive function $\rho$. 
 Since $|e(x)|=1$, then $e(x)\cdot e_{x_i}(x)=0$, for $i=1,2$, and
therefore 
\begin{equation}\label{eq:formulaforexidotz(x),u(z(x))}
e_{x_i}\cdot (\varphi(x),u(\varphi(x)))=(x,0)\cdot e_{x_i}(x);\qquad i=1,2.
\end{equation}
Hence
\begin{align*}
\left(e(x)-\kappa_1\kappa_2 \,w\right)\cdot  \left(\varphi(x),u(\varphi(x))\right)_{x_i}
&=G_{x_i}(x) - e_{x_i}(x)\cdot (x,0)\\
&=G_{x_i}(x) - [e(x)\cdot (x,0)]_{x_i} + e(x)\cdot (x,0)_{x_i}.
\end{align*}
Therefore \eqref{eq:odeforfarfield} and \eqref{eq:odeforfarfieldv} become
\begin{align*}
H_{x_1}(x)&=-G_{x_1}(x)+ [e(x)\cdot (x,0)]_{x_1} - e(x)\cdot (x,0)_{x_1}\\
H_{x_2}(x)&=-G_{x_2}(x)+ [e(x)\cdot (x,0)]_{x_2} - e(x)\cdot (x,0)_{x_2},
\end{align*} 
that is,
\begin{align}
[H+G-e(x)\cdot (x,0)]_{x_1}(x_1,x_2)&=-e_1(x_1,x_2)\label{eq:firstequationcurl}\\
[H+G-e(x)\cdot (x,0)]_{x_2}(x_1,x_2)&=-e_2(x_1,x_2)\label{eq:secondequationcurl}.
\end{align}
Since $e\in C^1$, this implies that the mixed derivatives of $H+G-e(x)\cdot (x,0)$ exist and are continuous.
From \cite[Theorem 9.41]{rudinbook:principlesofmathanalysis}, these mixed derivatives are equal.
This implies that
\begin{equation}\label{eq:curlzerocondition}
\text{curl } (e'(x))= \left(e_2(x)\right)_{x_1}-\left(e_1(x)\right)_{x_2}=0,
\end{equation}
with $e'(x)=(e_1(x),e_2(x))$, where we use the 2 dimensional definition of curl.
Therefore there is a scalar function $h(x_1,x_2)$ such that 
\[
\nabla h(x_1,x_2)=e'(x_1,x_2),
\] 
($\Omega$ simply connected).
Then \eqref{eq:firstequationcurl} implies that
\[
H+G-e(x)\cdot (x,0)+h(x)=\psi(x_2),
\]
and from \eqref{eq:secondequationcurl} we get that $\psi'(x_2)=0$, and so
\[
H+G-e(x)\cdot (x,0)+h(x)=C,
\]
with $C$ a constant. Hence  
\begin{equation}\label{eq:formulaford(x)forgeneralfield}
d(x)=d(x,C,w):=\dfrac{C-h(x)+e(x)\cdot (x,0)-\left[e(x)-\kappa_1\kappa_2 \,w\right]\cdot  \left(\varphi(x),u(\varphi(x))\right)}{ \kappa_1- \kappa_2\,w\cdot (e(x)-\lambda_1 \,\nu_1(x)) }.
\end{equation}
Notice that $\kappa_1- \kappa_2\,w\cdot (e(x)-\lambda_1 \,\nu_1)= \kappa_1-\kappa_1\kappa_2 m_{1}\cdot w\geq \kappa_1-\kappa_1\kappa_2>0$, since $n_{2}>n_{3}$. $d(x,C,w)$ represents the thickness of the lens along the direction $m_1(x)$. We obtained then a family of lenses with thickness depending on the choice of the constant $C$. This is studied more explicitly for magnifying lenses in Section \ref{sec:Magnification}.

At this point we have proved that if $f$ is given by \eqref{eq:formulafortherightcurveefield}, with $d\in C^1$, and satisfying 
the pdes \eqref{eq1} and \eqref{eq2}, then $\text{curl } (e_1(x_1,x_2),e_2(x_1,x_2))=0$, and $d$ is given by the formula \eqref{eq:formulaford(x)forgeneralfield} with $\nabla h=(e_1,e_2)$.
Reciprocally, if $u\in C^2$, $\varphi,e\in C^1$, with $\text{curl }(e_1(x),e_2(x))=0$ in $\Omega$ simply connected,
and $d$ is given by \eqref{eq:formulaford(x)forgeneralfield}, with $\nabla h=(e_1,e_2)$, then $d\in C^1$ and  
the surface $\sigma_2$ parametrized by  
\[
f(x)=\left(\varphi(x),u(\varphi(x))\right)+d(x)\,m_1(x),\quad \text{where }m_1(x)=\dfrac{1}{\kappa_1}(e(x)-\lambda_1\,\nu_1)
\]
satisfies the equations \eqref{eq1} and \eqref{eq2}.\footnote{The curl condition is very natural because the directions of light rays $s$ (defined as the direction of the average Poynting vector) are parallel to the gradient $\nabla S$ where $S$ is the wave front, i.e.,
$S$ is solution to the eikonal equation and $\text{curl } \nabla S=0$, \cite[Section 3.1.2]{book:born-wolf}.
}
This implies that the vector $m_1(x)-\kappa_2w$ is parallel to the normal to $f(x)$, and therefore the surface described by $f(x)$ refracts the incident vector $m_1(x)$ into the direction $w$.

So far we obtained that $f$ is a $C^1$ parametrization of the surface $\sigma_2$. We claim that $\sigma_2$ is in fact twice differentiable. To prove this we will show that the unit normal vector $\nu_2$ to $\sigma_2$ is $C^1$. 
By \eqref{eq:refractionatsigma2} we have that $\nu_2=\dfrac{1}{\lambda_2} (m_1-\kappa_2 w)$, then 
from (C1) it follows that
$$0<\lambda_2=|\lambda_2|=|m_1-\kappa_2 w|=\sqrt{1-2\kappa_2m_1\cdot w+\kappa_2^2}.$$
It is then enough to show that $m_1$ is $C^1$. We have that 
$$\lambda_1=e(x) \cdot \nu_1(x) -\kappa_1\sqrt{1-\kappa_1^{-2}\left(1-(e(x)\cdot \nu_1(x)\right)^2)}.$$
Since $\kappa_1>1$, then $1-\kappa_1^{-2}(1-(e(x)\cdot \nu_1(x))^2)>0$. 
We have $e\in C^1$ and $\nu_1\in C^1$ (since $\sigma_1\in C^2$) then we get $\lambda_1(x)\in C^1$. 
Thus, $m_1 \in C^1$ from \eqref{eq:formula2form1}.

This completes the proof of the theorem.
\end{proof}

\begin{remark}\rm
In the 2d case, that is, when $\Omega$ is an interval, $\sigma_1$ and $\sigma_2$ are curves, and $e(x)=(e_1(x),e_2(x))$ is a two dimensional unit field, then Theorem \ref{thm:existenceofupperlenssurfaces} holds 
without any conditions on the derivative of $e_1$.
\end{remark}

\begin{remark}\rm
Combining lenses of the type described and using reversibility of optical paths, one can construct optical systems composed of two of these lenses that refract a given incoming field of directions $e(x)$ into another given field of directions $g(x)$.\end{remark}

\begin{remark}\label{rmk:self-interesection}\rm
Since the surface $\sigma_2$ in Theorem \ref{thm:existenceofupperlenssurfaces} is given parametrically, it may have self-intersections. This depends on the constant $C$ in \eqref{eq:formulaford(x)forgeneralfield}, the values of $\kappa_1,\kappa_2$, and the size of the Lipschitz constants for the functions $e,u,Du$ and $\varphi$, and the domains $\Omega$ and $\Omega'$. Clearly, if $\sigma_2$ has self-intersections, then it is not physically realizable. 
This is analyzed in detail in the forthcoming paper \cite{gutierrez-sabra:farfield-with-two-surfaces}. However, to illustrate this issue, for simplicity we analyze here the case when $e(x)=w=e_3$, that is, the incoming and outgoing rays are collimated; here we have $\varphi(x)=x$.
In this case, one can prove the following Lipschitz estimate for the refracted direction $m_1(x)$:
\begin{equation}\label{eq:Lipestimateformm_1}
|m_1(x)-m_1(y)|
\leq 
M(\kappa_1)\,|Du(x)-Du(y)|,
\end{equation}
for all $x,y\in \Omega$ and $M(\kappa_1)$ a constant depending only on $\kappa_1$. This estimate follows by calculation from the following estimate for the normal:
\[
|\nu_1(x)-\nu_1(y)|\leq \sqrt5\, |Du(x)-Du(y)|.
\]
This implies a Lipschitz estimate, linear in $|C|$, for the distance function $d(x,C,w)$.
In fact, we can write
$d(x,C,w)=\dfrac{C+v(x)}{g(x)}$, with
$v(x)=-(1-\kappa_1\kappa_2)\,u(x)$ and
$g(x)=\kappa_1-\kappa_1\kappa_2\,w\cdot m_1(x)$.
Since $\kappa_1(1-\kappa_2)\leq g(x)\leq \kappa_1(1+\kappa_2)$, we have
\begin{align*}
\left| d(x,C,w)-d(y,C,w)\right|
&=
\left| \dfrac{C\,(g(y)-g(x))+v(x)g(y)-v(y)g(x)}{g(x)g(y)}\right|\\
&\leq
\dfrac{1}{\kappa_1^2(1-\kappa_2)^2}\left(g(y)\,\left|v(x)-v(y) \right|+(|C|+|v(y)|)\left|g(y)-g(x) \right|\right).
\end{align*}
From \eqref{eq:Lipestimateformm_1}
\[
|g(x)-g(y)|\leq \kappa_1\kappa_2\,|m_1(x)-m_1(y)|
\leq M_1(\kappa_1,\kappa_2)\, |Du(x)-Du(y)|\leq M_1(\kappa_1,\kappa_2)\,L_{Du} |x-y|,
\]
with $L_{Du}$ the Lipschitz constant of $Du$ in $\Omega$.
Also
\[
|v(x)-v(y)|\leq (1+\kappa_1\kappa_2)\,|u(x)-u(y)|\leq (1+\kappa_1\kappa_2)\,L_u\,|x-y|,
\]
where $L_u$ is the Lipschitz constant of $u$ in $\Omega$; and $|v(y)|\leq (1+\kappa_1\kappa_2)\max_{\Omega}u$.
Therefore
\begin{equation}\label{eq:Lipschitzestimateofdx,C,w}
\left| d(x,C,w)-d(y,C,w)\right|
\leq
\left(A(\kappa_1,\kappa_2)\,L_u+B(\kappa_1,\kappa_2)\left[|C|+(1+\kappa_1\kappa_2)\max_{\Omega}u\right]\,L_{Du}\right)\,|x-y|.
\end{equation}
With this estimate we will show that if $C$ is appropriate, then the surface $\sigma_2$ given by $f(x)=(x,u(x))+d(x,C,w)\,m_1(x)$ cannot have self-intersections, i.e., $f$ is injective.
In fact, suppose that there are two points $x\neq y$ such that $f(x)=f(y)$. Then
\begin{align*}
|(x,u(x))-(y,u(y))|&\leq
d(y,C,w)\,|m_1(y)- m_1(x)|+|d(y,C,w)\,m_1(x)-d(x,C,w)\,m_1(x)|\\
&\leq
d(y,C,w)\,|m_1(y)- m_1(x)|+|d(y,C,w)-d(x,C,w)|.
\end{align*}
We have
\[
d(y,C,w)\leq \dfrac{|C|+(1+\kappa_1\kappa_2)\max_{\Omega}u}{\kappa_1(1-\kappa_2)}.
\] 
We then obtain from \eqref{eq:Lipestimateformm_1} and \eqref{eq:Lipschitzestimateofdx,C,w} that
\begin{align*}
1&\leq \dfrac{|C|+(1+\kappa_1\kappa_2)\max_{\Omega}u}{\kappa_1(1-\kappa_2)}\,M(\kappa_1)\,L_{Du}
+
A(\kappa_1,\kappa_2)\,L_u+B(\kappa_1,\kappa_2)\left[|C|+(1+\kappa_1\kappa_2)\max_{\Omega}u\right]\,L_{Du}\\
&=
\left(\dfrac{M(\kappa_1)}{\kappa_1(1-\kappa_2)}+B(\kappa_1,\kappa_2) \right)\,L_{Du}\,|C|
+
(1+\kappa_1\kappa_2)\max_{\Omega}u\left( \dfrac{M(\kappa_1)}{\kappa_1(1-\kappa_2)}\,+B(\kappa_1,\kappa_2)\right)\,L_{Du}
+A(\kappa_1,\kappa_2)\,L_u\\
&:=\alpha\,|C|+\beta.
\end{align*}
If we choose $L_u$ and $L_{Du}$ sufficiently small, then $\beta<1/2$. This implies that $|C|\geq \dfrac{1}{2\,\alpha}$.
Therefore if $|C|<1/2\alpha$, the surface $\sigma_2$ cannot have self intersections.
Let $\bar C=\max_\Omega -v(x)$. If we choose $L_{Du}$ small enough such that $\bar C<\frac{1}{2\alpha}$ and pick $C>\bar C$ with $|C|<\frac{1}{2\alpha}$, then $d(x,C,w)>0$ and the surface $\sigma_2$ is physically realizable.
\end{remark}

\subsection{Example 1.}\label{sec:Pointsource} 
%
%
%
%
%
%

%
%
%
%
The case with one point source considered in \cite[Section 3]{gutierrez:asphericallensdesign} is a special case of the problem considered above.
In fact, suppose that the rays with direction $e(x)$ for $x\in \Omega$ all intersect at a virtual point $V=(a,b,c)$ below the plane containing $\Omega$. Then $e(x)=\dfrac{(x,0)-V}{|(x,0)-V|}$, and so $e'(x)=\dfrac{x-(a,b)}{|(x,0)-V|}$ which is equal to $\nabla_x|(x,0)-V|$ and therefore $\text{curl } e'(x)=0$ in $\Omega$.

\subsection{Example 2.}\label{sec:avoidreflection}

We are given the incident unit field $e(x)=(e_1(x),e_2(x),e_3(x))=(e'(x),e_3(x))$ with $x\in \Omega,$ where $e'(x)=\nabla h(x)$, and $e_3(x)>0$. First, we are going to construct a surface $\sigma_1$ that is orthogonal to the incident rays, and consequently minimizes internal reflection. Then using Theorem \ref{thm:existenceofupperlenssurfaces}, we find the surface $\sigma_2$ such that the lens enclosed by $\sigma_1$ and $\sigma_2$ refracts all the field $e(x)$ into a given unit direction $w$.
The surface $\sigma_1$ is parametrized by the vector 
$$
g(x)=(x,0)+\lambda(x)\, e(x)=(x+\lambda(x)\,e'(x), \lambda(x)\,e_3(x)),$$
with $\lambda(x)$ a positive function to be determined so that $e(x)$ is normal to $g(x)$ for each $x\in \Omega$. This is equivalent to find $\lambda(x)$ so that $e(x)$ is perpendicular to the tangent plane at $g(x)$, that is,
$$
\left((x,0)+\lambda(x)\, e(x)\right)_{x_i}\cdot e(x)=0,
$$
for $i=1,2$.
Which is in turn equivalent to
$e_{i}(x)+\lambda_{x_i}(x)=0$, for $i=1,2$, since $|e(x)|=1$. Therefore, if we choose $\lambda(x)=-h(x)+\tilde C$, with $\tilde C$ an appropriate constant so that $\lambda(x)>0$ for all $x\in \Omega$, we then obtain $\sigma_1$ orthogonal to the field $e(x)$.

To apply Theorem \ref{thm:existenceofupperlenssurfaces}, we need to write 
$g(x)=(\varphi(x),u(\varphi(x)))$ with $\varphi(x)=x+\lambda(x)\,e'(x)$, and $u$ such that $u(\varphi(x))=\lambda(x)\,e_3(x)$. By the inverse theorem, this is possible if the Jacobian of the map $\varphi$ is not zero. In such case, we let $u(z)=\lambda\left(\varphi^{-1}(z)\right)e_3\left(\varphi^{-1}(z)\right).$ Since $\nu_1(x)=e(x)$, then from \eqref{eq:formulaforlambda}, $\lambda_1(x)=1-\kappa_1$, from \eqref{eq:formula2form1} $m_1(x)=e(x)$, and therefore from \eqref{eq:formulaford(x)forgeneralfield} we obtain
$$
d(x)=\dfrac{C-\tilde C+\kappa_1\kappa_2\,w\cdot \left((x,0)+\lambda(x)\,e(x) \right)}{\kappa_1-\kappa_1\kappa_2w\cdot e(x)}.
$$
Then \eqref{eq:formulafortherightcurveefield} yields the following simplified 
parametrization for $\sigma_2$
\[
f(x)=(x,0)+\left(\lambda(x)+d(x) \right)\,e(x)
=(x,0)+\dfrac{C-\tilde C +\kappa_1\lambda(x)+\kappa_1\kappa_2w\cdot (x,0)}{\kappa_1-\kappa_1\kappa_2w\cdot e(x)}e(x).\]

\subsection{Example 3.}\label{sec:verticalcase} Consider the special case when $e(x)=w=(0,0,1)$, i.e., the rays entering and leaving the lens have unit vertical direction $e_3$.  Clearly, in this case the curl condition \eqref{eq:curlzerocondition} is satisfied, and $\varphi(x)=x$.
So the ray from $(x,0)$ strikes the surface $\sigma_1$ at the point $P(x)=(x,u(x))$.
The normal at $P(x)$ is given by
$$\nu_1(x)=\dfrac{(-Du(x),1)}{\sqrt{1+|Du(x)|^2}},$$ and so the condition $\nu_1(x)\cdot e_3\geq 0$ is satisfied. 

%

For the application that will be described later on, we now rewrite $d(x)$ and $f(x)$ in terms of $u$ and its gradient. 
We have that
\begin{align*}
\lambda_1&= e_3\cdot \nu_1 -\kappa_1 \,\sqrt{1-\dfrac{1}{\kappa_1^2}\left(1-(e_3\cdot \nu_1)^{2}\right)}\\
&=\dfrac{1}{\sqrt{1+|Du(x)|^2}}\left(1-\kappa_1\sqrt{1+\left(1-\dfrac{1}{\kappa_1^2}\right)|Du(x)|^2}\right),
\end{align*}
and then
$$\lambda_1\nu_1(x)=\Delta(x)\left(-Du(x),1\right),$$ with

\begin{equation}\label{eq:definitionofDelta(x)}
\Delta(x):=\dfrac{1}{1+|Du(x)|^2}
\left(1-\kappa_1\sqrt{1+\left( 1-\dfrac{1}{\kappa_1^2}\right)|Du(x)|^2} \right).
\end{equation}

Condition \eqref{eq:conditionforrefraction} becomes
\begin{equation}\label{eq:conditionvertical}
\Delta(x)\leq 1-\kappa_1\kappa_2.
\end{equation}

We conclude that if $u$ satisfies \eqref{eq:conditionvertical}, then by \eqref{eq:formulafortherightcurveefield}, \eqref{eq:formula2form1}, and \eqref{eq:formulaford(x)forgeneralfield} 
\begin{align}
f(x)&=(x,u(x))+d(x)m_1(x), \label{eq:generalformulaforf}\\
d(x)&= -\dfrac{(1-\kappa_1\kappa_2)u(x)+C}{\kappa_1-\kappa_2(1-\Delta(x))},\label{eq:generalformulaford}\\
m_1(x)&=\dfrac{1}{\kappa_1}\left( \Delta(x)\,Du(x), 1-\Delta(x) \right). \label{eq:formula3form1}
\end{align}

%

\setcounter{equation}{0}
\section{Imaging Problem}\label{sec:nearfield}
We are given a bijective transformation $T$ between the planar source $\Omega$ and a plane domain $\Omega^*$, and $a>0$. 
Our goal now is to find two surfaces $\sigma_1$ and $\sigma_2$ such that the lens sandwiched by them refracts every ray emitted from $(x,0)$ with vertical direction $e_3$ into the point $(Tx,a)$, and such that rays leaving $\sigma_2$ have direction $e_3$. 
Using the construction from the previous section, $\sigma_2$ is determined by $\sigma_1$ which has the form $(x,u(x))$ with $u\in C^2$ to be calculated.
\begin{figure}
\includegraphics[width=4in]{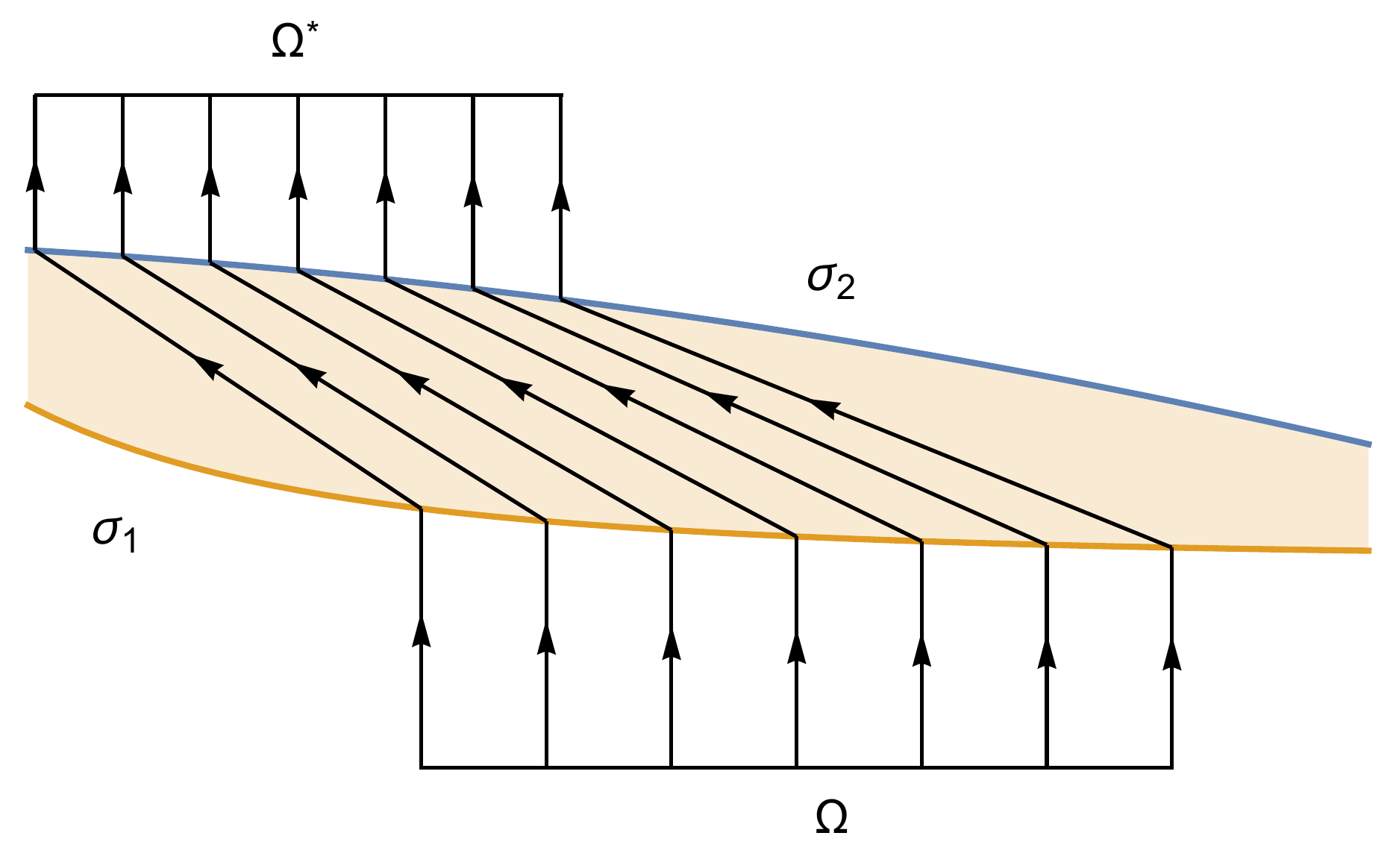}
\caption{Imaging problem}
\label{fig:imagingproblem}
\end{figure}

More precisely, the ray with direction $e_3$ strikes $\sigma_1$ at $P(x)=(x,u(x))$ and is refracted into the unit direction $m_1$ given by \eqref{eq:formula3form1}. The ray then hits $\sigma_2$ at the point $f(x)=(f_1(x),f_2(x),f_3(x))$, given in terms of $u(x)$ by \eqref{eq:generalformulaforf}. It is then finally refracted into the medium $n_3$ with direction $e_3$ into the point $(Tx,a)$, see Figure \ref{fig:imagingproblem}. 
As before, we assume that the material enclosed within the lens has refractive index $n_2$, such that $n_2>n_1, n_3$, but $n_1$ and $n_3$ are unrelated; and we use the notation $\kappa_1=n_2/n_1, \kappa_2=n_3/n_2$.

Since rays leave $\sigma_2$ with direction $e_3$ into the point $(Tx,a)$, then $Tx=(f_1(x),f_2(x))$. Therefore from \eqref{eq:generalformulaforf}, \eqref{eq:generalformulaford}, \eqref{eq:formula3form1}, we get that  $u$ must satisfy the following  system of partial differential equations
\begin{equation}\label{eq:firstformpde}
Tx-x=-\dfrac{(1-\kappa_1\kappa_2)u(x)+C}{\kappa_1^2-\kappa_1\kappa_2(1-\Delta(x))}\Delta(x)Du(x),
\end{equation}
where $\Delta(x)$ is given by \eqref{eq:definitionofDelta(x)}.
To write \eqref{eq:firstformpde} more explicitly, notice that
\begin{equation}\label{eq:specialformula}
\Delta(x)\left(1+\sqrt{\kappa_1^2+(\kappa_1^2-1)|Du(x)|^2}\right)=1-\kappa_1^2.
\end{equation}
Multiplying and dividing the right hand side of \eqref{eq:firstformpde} by $\left(1+\sqrt{\kappa_1^2+(\kappa_1^2-1)|Du(x)|^2}\right)$, and using \eqref{eq:specialformula}, we obtain that
the pde \eqref{eq:firstformpde} can be re written as
\begin{equation}\label{eq:secondformpde}
\dfrac{(1-\kappa_1\kappa_2)u(x)+C}{(\kappa_1^2-\kappa_1\kappa_2)\sqrt{\kappa_1^2+(\kappa_1^2-1)|Du(x)|^2}+\kappa_1^2(1-\kappa_1\kappa_2)}Du(x)=\dfrac{Sx}{\kappa_1^2-1},
\end{equation}
where $Sx=Tx-x$.

Since $\kappa_1>1$, then from \eqref{eq:specialformula} $\Delta(x)<0$. 
Hence the refraction condition \eqref{eq:conditionvertical} follows immediately if $\kappa_1\kappa_2\leq 1$, that is, when $n_1\geq n_3$.

\setcounter{equation}{0}
\section{Solution of the imaging problem when $n_1=n_3$}\label{sec:sameindex}

In this section we prove the following theorem.
\begin{theorem}\label{thm:resultwhensameindex}
Assume $n_1=n_3$. Let $T$ be a $C^1$ bijective map between the set $\Omega$ and a plane domain $\Omega^*$, and $a>0$. Then there exists a lens refracting all rays emitted from $(x,0), \, x\in \Omega$, with direction $e_3$ into the point $(Tx,a)$ such that rays leave the lens with direction $e_3$ if and only if
$$\left(C^2-(\kappa_1^2-1)\,|S|^2\right)\,\text{\rm curl } S=\dfrac{\kappa_1^2-1}{2}\,S\times D|S|^2,$$
where $Sx=Tx-x$, and $C$ is a negative constant with $|C|>|Sx|\sqrt{\kappa_1^2-1}$.
In this case, the lower face of the lens $\sigma_1$ is parametrized by $(x,u(x))$ where $u$ is given by \eqref{eq:solutionn1equaln3}, and the top face $\sigma_2$ of the lens is parametrized by the vector $f(x)=(x,u(x))+d(x)m_1(x)$, with $d$ and $m_1$ given by \eqref{eq:formulawhenn3equal} and \eqref{eq:formula3form1} respectively. Both surfaces $\sigma_1$ and $\sigma_2$ are $C^2$.
\end{theorem}

\begin{proof}
In this case $\kappa_1\kappa_2=1$ and then by \eqref{eq:generalformulaford}, \eqref{eq:formula3form1} we have that
\begin{equation}\label{eq:formulawhenn3equal}
d(x)=-\dfrac{C}{\kappa_1-e_3\cdot m_1}.
\end{equation}
Since $d(x)>0$ and $e_3\cdot m_1 \leq 1<\kappa_1$, then $C<0$. 
The pde \eqref{eq:secondformpde} becomes
\begin{equation}\label{eq:pdewhenn1equaln3}
\dfrac{Du(x)}{\sqrt{\kappa_1^2+(\kappa_1^2-1)|Du(x)|^2}}=\dfrac{Sx}{C}.
\end{equation}
Taking absolute values it follows that 
\begin{equation}\label{eq:boundforTminusI}
|Sx|<\dfrac{|C|}{\sqrt{\kappa_1^2-1}}.
\end{equation}
Therefore to find $u$ solving \eqref{eq:pdewhenn1equaln3}, the constant $C$, and the map $T$ must be chosen so that condition \eqref{eq:boundforTminusI} is satisfied for all $x\in \Omega$. 
When $T$ is a magnification, condition \eqref{eq:boundforTminusI} is related to the thickness of the lens which will be analyzed in Example \ref{sec:Magnification}.

To solve the pde \eqref{eq:pdewhenn1equaln3}, an extra condition on $S$ is required, condition \eqref{eq:conditionsamemedium}.
In fact, 
taking absolute values and squaring both sides of \eqref{eq:pdewhenn1equaln3} we obtain that
\[
|Du(x)|^2\left(C^2-\left(\kappa_1^2-1\right)|Sx|^2\right)=\kappa_1^2|Sx|^2.
\]
From \eqref{eq:boundforTminusI}, $C^2-\left(\kappa_1^2-1\right)|Sx|^2>0$, and so
\begin{equation*}
|Du(x)|= \dfrac{\kappa_1|Sx|}{\sqrt{C^2-\left(\kappa_1^2-1\right)|Sx|^2}},
\end{equation*}
which yields
\[
\kappa_1^2+\left(\kappa_1^2-1\right)|Du(x)|^2
=\dfrac{\kappa_1^2\,C^2}{C^2-\left(\kappa_1^2-1\right)|Sx|^2}.
\]
Replacing this in \eqref{eq:pdewhenn1equaln3} and using that $C<0$, we obtain the equivalent linear system
\begin{equation}\label{eq:pden1equaln3simplified}
Du(x)=\dfrac{-\kappa_1\,Sx}{\sqrt{C^2-\left(\kappa_1^2-1\right)|Sx|^2}}
:=
F(x)=(F_1(x),F_2(x)).
\end{equation}
If $u$ is a $C^2$ solution to \eqref{eq:pden1equaln3simplified}, then the field $F$ is conservative. Conversely if  $\partial_{x_2}F_1= \partial_{x_1}F_2$ in $\Omega$ with $\Omega$ simply connected, then 
\begin{equation}\label{eq:solutionn1equaln3}
u(x)=u(x_0)+\int_\gamma F(x)\cdot dr,
\end{equation}
is a solution to \eqref{eq:pdewhenn1equaln3}, where the integral is a line integral along an arbitrary path $\gamma$ from $x_0$ to $x$ in $\Omega$.

The equality $\partial_{x_2}F_1= \partial_{x_1}F_2$ yields a condition on $S=(S_1,S_2)$. In fact,
{\small
\begin{align*}
\partial_{x_2}F_1
&=-\kappa_1 \dfrac{\partial_{x_2}S_1 \,\sqrt{C^2-\left(\kappa_1^2-1\right)|S|^2}+\left(\kappa_1^2-1\right)\dfrac{S_1\,\partial_{x_2}S_1+S_2\,\partial_{x_2}S_2}{\sqrt{C^2-\left(\kappa_1^2-1\right)|Sx|^2}}S_1}{\left(\sqrt{C^2-\left(\kappa_1^2-1\right)|S|^2}\right)^2}\\
&= \dfrac{-\kappa_1}{\left(C^2-\left(\kappa_1^2-1\right)|S|^2\right)^{3/2}}\left[\left(C^2-\left(\kappa_1^2-1\right)S_2^2\right)\partial_{x_2}S_1+\left(\kappa_1^2-1\right)\,S_1S_2\,\partial_{x_2}S_2\right].
\end{align*}
}
Similarly
{\small
\[
\partial_{x_1}F_2=\dfrac{-\kappa_1}{\left(C^2-\left(\kappa_1^2-1\right)|S|^2\right)^{3/2}}\left[\left(C^2-\left(\kappa_1^2-1\right)S_1^2\right)\partial_{x_1}S_2+\left(\kappa_1^2-1\right)\,S_1S_2\,\partial_{x_1}S_1\right].
\]
}
Therefore $S$ must satisfy the following condition
{\small
\[
\left(C^2-\left(\kappa_1^2-1\right)S_2^2\right)\partial_{x_2}S_1+\left(\kappa_1^2-1\right)\,S_1S_2\,\partial_{x_2}S_2=
\left(C^2-\left(\kappa_1^2-1\right)S_1^2\right)\partial_{x_1}S_2+\left(\kappa_1^2-1\right)\,S_1S_2\,\partial_{x_1}S_1.
\]
}
Simplifying this expression yields
%
$$
C^2 \,\left(\dfrac{\partial S_2}{\partial x_1}- \dfrac{\partial S_1}{\partial x_2}\right)
+ (\kappa_1^2-1)
\left(S_1 S_2 \left( \dfrac{\partial S_1}{\partial x_1}-\dfrac{\partial S_2}{\partial x_2}\right)
+
S_2^2 \dfrac{\partial S_1}{\partial x_2}
-S_1^2 \dfrac{\partial S_2}{\partial x_1}\right)=0,
$$
and if we set $|S|=\sqrt{(S_1x)^2 +(S_2x)^2}$, then the last condition is equivalent to
\begin{equation}\label{eq:conditionsamemedium}
\left(C^2-(\kappa_1^2-1)\,|S|^2\right)\,\text{curl}\,S=\dfrac{\kappa_1^2-1}{2}\,S\times D|S|^2,
\end{equation}
where $\times$ denotes the cross product in two dimensions. 
This completes the proof of the theorem.
\end{proof}

Notice that if $Tx=((\alpha_1+1)x_1,(\alpha_2+1)x_2)$, then $Sx= (\alpha_1\,x_1 ,\alpha_2\, x_2)$, $\text{curl } S=0$ and
\begin{align*}
S\times D|S|^2&= (\alpha_1\, x_1,\alpha_2\, x_2)\times (2\alpha_1x_1,2\alpha_2 x_2)\\
&=2\alpha_1\alpha_2x_1x_2(\alpha_1-\alpha_2).
\end{align*}
So $S$ satisfies condition \eqref{eq:conditionsamemedium} if and only if $\alpha_1=\alpha_2$ or $\alpha_1=0$ or 
$\alpha_2=0$. In the following examples we analyze these cases.

\subsection{Example 1.}\label{sec:Magnification} $Tx=(1+\alpha)x$, $\alpha> 0$ i.e.
 $T$ is a magnification with factor $1+\alpha$. In this case, we have that $F$ in \eqref{eq:pden1equaln3simplified} is
\[
F(x)=\dfrac{-\kappa_1\,\alpha \, x}{\sqrt{C^2-\left(\kappa_1^2-1\right)\alpha^2|x|^2}},
\]
and \eqref{eq:pden1equaln3simplified} then reads
\begin{align}
u_{x_1}(x)&= \dfrac{-\kappa_1\,\alpha\,x_1}{\sqrt{C^2-\left(\kappa_1^2-1\right)\alpha^2|x|^2}} \label{eq:firstpartialdilation}\\
u_{x_2}(x)&=\dfrac{-\kappa_1\,\alpha\,x_1}{\sqrt{C^2-\left(\kappa_1^2-1\right)\alpha^2|x|^2}}. \label{eq:secondpartialdilation}
\end{align}
Integrating \eqref{eq:firstpartialdilation} with respect to $x_1$ yields
\[
u(x)= \dfrac{\kappa_1}{\alpha\left(\kappa_1^2-1\right)}\sqrt{C^2-\left(\kappa_1^2-1\right)\alpha^2|x|^2}+\phi(x_2).
\]
Differentiating this with respect to $x_2$ and using \eqref{eq:secondpartialdilation} we obtain that $\phi'(x_2)=0$ and hence
\begin{equation}\label{eq:solutiondilation}
u(x)=\dfrac{\kappa_1 }{\alpha(\kappa_1^2-1)}\sqrt{C^2-\left(\kappa_1^2-1\right)\alpha^2|x|^2}+A.
\end{equation}
Notice that \eqref{eq:solutiondilation} implies that the graph of $u$ is contained in the ellipsoid with equation $(z-A)^2+\kappa_1^2\,|x|^2=\left(\dfrac{C\,\kappa_1}{\alpha\,(\kappa_1^2-1)} \right)^2$.
Figure \ref{fig:picnearfield} shows the lens obtained when $Tx=2x, a=6, n_1=n_3=1$, and $n_2=1.52$.

\begin{figure}
\includegraphics[width=3in, height =3in]{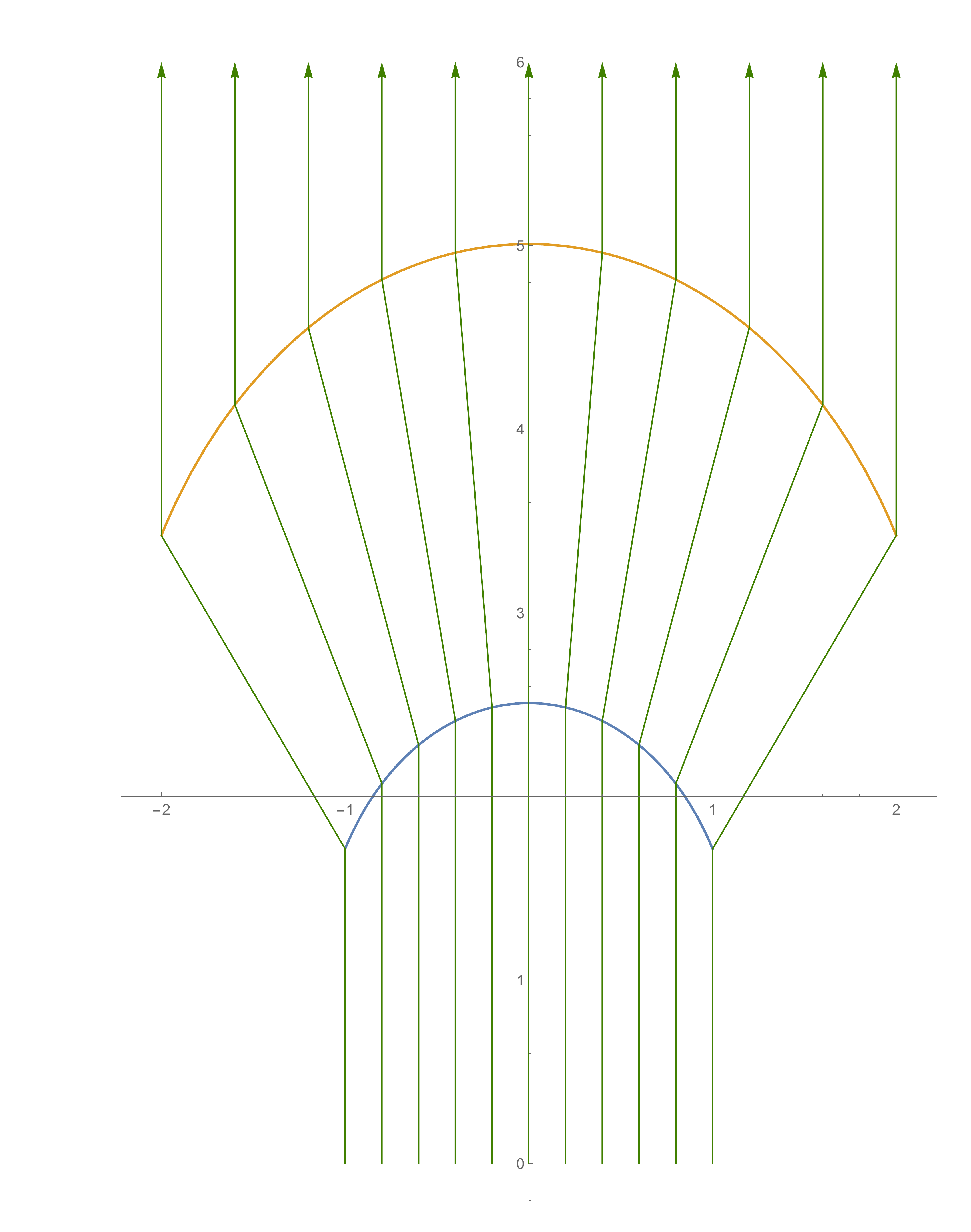}
\caption{Near Field magnification}
\label{fig:picnearfield}
\end{figure}

We analyze now the thickness of the lens constructed.

If $\Omega$ is the closed ball  $B_\gamma(0)$, then condition \eqref{eq:boundforTminusI} is satisfied for $x\in B_\gamma(0)$ when
\begin{equation}\label{eq:boundforTminusImagnification}
\alpha \gamma< \dfrac{|C|}{\sqrt{\kappa_1^2-1}}.
\end{equation} 
To analyze the thickness, 
we fix the value of $u$ at a point, say at $x=0$, and its normal vector $\nu_0$ at that point. The refracted vector $m_1(0)$ at $(0,u(0))$ is determined by \eqref{eq:formula3form1}.
From \eqref{eq:formulawhenn3equal} 
\begin{equation}\label{eq:formulaforconstantC}
C=-d(0)\,\left( \kappa_1-e_3\cdot m_1(0)\right).
\end{equation}
$d(0)$ represents the distance along the ray with direction $m_1(0)$ from the point $(0,u(0))$ to $f(0)$ on the second surface.
This means that once we prescribe the value of $d(0)$, that is, the thickness of the lens on the refracted ray with direction $m_1(0)$, the value of $C$ is given by \eqref{eq:formulaforconstantC}. Notice that the value of the constant $C$ depends also on the value of the normal $\nu_0$.

Since $n_1<n_2$, by (C2) applied to $\sigma_1$, we have $e_3\cdot m_1\geq \dfrac{1}{\kappa_1}$. Therefore, 
$$\kappa_1-1\leq \kappa_1-e_3\cdot m_1(x)\leq \kappa_1-\dfrac{1}{\kappa_1}.$$ 
For $x=0$, since $d(0)>0$, from \eqref{eq:formulaforconstantC} we then obtain the following bounds for the constant $C$:
\begin{equation}\label{eq:boundsfor|C|}
d(0)\left( \kappa_1-1\right)\leq |C|\leq d(0)\left( \kappa_1-\dfrac{1}{\kappa_1}\right).
\end{equation}
Notice that from \eqref{eq:formulawhenn3equal}, $C=-d(x)\,(\kappa_1-e_3\cdot m_1(x))$ for all $x\in \Omega$. Then from \eqref{eq:boundsfor|C|}, we obtain the following bound for the thickness $d(x)$ for each $x\in B_\gamma(0)$
\[
d(0)\,\dfrac{\kappa_1}{\kappa_1+1}\leq d(x)\leq d(0)\,\dfrac{\kappa_1+1}{\kappa_1}.
\]

Now if we assume \eqref{eq:boundforTminusImagnification}, we then obtain from \eqref{eq:boundsfor|C|} that
\[
\alpha\,\gamma\leq d(0)\,\dfrac{\sqrt{\kappa_1^2-1}}{\kappa_1}.
\]
This means that if we choose $\alpha$ and $\gamma$ arbitrarily, to have the desired refraction job we need to take the thickness $d(0)$ sufficiently large and $C$ satisfying \eqref{eq:formulaforconstantC}.
On the other hand, if we choose in advance the thickness $d(0)$, and then pick $\alpha,\gamma$ satisfying 
\[
\alpha \,\gamma<d(0)\,\dfrac{\kappa_1-1}{\sqrt{\kappa_1^2-1}},
\]
then choosing $C$ in accordance with \eqref{eq:formulaforconstantC} we obtain, in view of \eqref{eq:boundsfor|C|}, that 
\eqref{eq:boundforTminusImagnification} holds and therefore the solution $u$ is given by \eqref{eq:solutiondilation} with the constant $A$ determined by the value of $u(0)$.

\subsection{Example 2}
If $Tx=((1+\alpha_1)x_1,x_2)$, then $F$ in \eqref{eq:pden1equaln3simplified} is
  \[
F(x)=\dfrac{-\kappa_1 (\alpha x_1,0)}{\sqrt{C^2-\left(\kappa_1^2-1\right)\alpha^2 x_1^2}},
\]
and $T$ satisfies \eqref{eq:boundforTminusI} when $|\alpha_1 \,x_1|\leq |C|/\sqrt{\kappa_1^2-1}$.
Then solving  \eqref{eq:pden1equaln3simplified} yields 
\begin{equation}\label{eq:onevariablesolutiondilation}
u(x_1,x_2)= \dfrac{\kappa_1}{\alpha \left(\kappa_1^2-1\right)}\sqrt{C^2-\left(\kappa_1^2-1\right)\alpha^2 x_1^2}+A.
\end{equation}

More generally, if $Tx=(h(x_1),x_2)$, then $Sx=(h(x_1)-x_1,0)$ satisfies \eqref{eq:conditionsamemedium}. If $|h(x_1)-x_1|\leq |C|/\sqrt{\kappa_1^2-1}$, then solving \eqref{eq:pden1equaln3simplified} we obtain for $x_0,x\in \Omega$ that
\begin{equation}\label{eq:onevariablesolutiontransformation}
u(x_1,x_2)= u(x_0^1,x_0^2)+ \int_{x_0^1}^{x_1} \dfrac{-\kappa_1(h(t)-t)}{\sqrt{C^2-(\kappa_1^2-1)(h(t)-t)^2}}\, dt.
\end{equation}

\setcounter{equation}{0}
\section{Solution of the imaging problem when $n_1>n_3$ }\label{sec:n3smaller}

In this case $\kappa_1\kappa_2<1$. We are seeking for a solution $u$ to \eqref{eq:secondformpde} so that the lens $(\sigma_1,\sigma_2)$ solves the problem described in Section \ref{sec:nearfield}, where $\sigma_1$ is given by the graph of $u$ and $\sigma_2$ is parametrized by the vector $f(x)=(x,u(x))+d(x)m_1$ where $d$ and $m_1$ are given in \eqref{eq:generalformulaford}, and \eqref{eq:formula3form1} respectively. This case have potential applications to design underwater vision devices
\cite{patent-air-water-camera}, and in immersion microscopy \cite{glycerol-lens-leica} and \cite{water-inmmersion-objective}.

 We remark the following. From \eqref{eq:specialformula} and since $\kappa_1\kappa_2<1$, we deduce that the denominator in formula \eqref{eq:generalformulaford} is positive. Since $d(x)$ must be positive, it follows that
the numerator $(1-\kappa_1\kappa_2)u(x)+C$ must be negative. Moreover, since the lens we want to construct is above the source, then $\sigma_1$ is above the plane containing $\Omega$, i.e., $u$ must be positive.

We will prove in this section the following theorem.

\begin{theorem}\label{eq:existensolutionpde}
Let $C$ be a negative constant, $x_0\in \Omega$, with $\Omega,\Omega^*$ plane domains with $\Omega$ simply connected, $T:\Omega\to \Omega^*$ bijective and $C^1$, and $Sx=Tx-x$. 
Suppose that
\begin{equation}\label{eq:conditionatx_0}
|Sx_0|<\dfrac{-C\,\sqrt{\kappa_1^2-1}}{\kappa_1\,(\kappa_1-\kappa_2)},
\end{equation}
and 
\begin{align}
\text{\rm curl } Sx &=0\label{firstcurlconditionforsolution}\\
S\times D(|Sx|^2)&=0\label{eq:secondcurlconditionforsolution},
\end{align}
for all $x$ in a neighborhood of $x_0$; see Remark \ref{rmk:condtionsonS}.

Then given 
\begin{equation}\label{eq:initialconditiondelta0}
0<\delta_0<\dfrac{1}{1-\kappa_1\kappa_2}\left(-C-\dfrac{\kappa_1(\kappa_1-\kappa_2)}{\sqrt{\kappa_1^2-1}} |Sx_0|\right)
\end{equation} 
there exists a neighborhood $U$ of $x_0$, and a unique $u\in C^2(U)$ solving the PDE \eqref{eq:secondformpde} in $U$ and  satisfying $u(x_0)=\delta_0$ and 
\begin{equation}\label{eq:inequalityinu}
0<u(x)<\dfrac{-C}{1-\kappa_1\kappa_2},
\end{equation}
for every $x\in U$.
\end{theorem}


\begin{proof}
We rewrite \eqref{eq:secondformpde} in the following form
{\small
\[
\dfrac{\left(u(x)+\dfrac{C}{1-\kappa_1\kappa_2}\right)(\kappa_1-\kappa_2)\sqrt{\kappa_1^2-1}}{\sqrt{\kappa_1^2(\kappa_1-\kappa_2)^2+(\kappa_1-\kappa_2)^2(\kappa_1^2-1)|Du(x)|^2}+\kappa_1(1-\kappa_1\kappa_2)}Du(x)(\kappa_1-\kappa_2)\sqrt{\kappa_1^2-1}=\dfrac{\kappa_1(\kappa_1-\kappa_2)^2}{(1-\kappa_1\kappa_2)}Sx.
\]
}
Let $v(x)=\left(u(x)+\dfrac{C}{1-\kappa_1\kappa_2}\right)(\kappa_1-\kappa_2)\sqrt{\kappa_1^2-1}$, $\bar Sx=\dfrac{\kappa_1(\kappa_1-\kappa_2)^2}{1-\kappa_1\kappa_2}Sx$. 
Then to find $u$ solving \eqref{eq:secondformpde}, satisfying \eqref{eq:inequalityinu}, with $u(x_0)=\delta_0$ is equivalent to find $v$ satisfying 
\begin{align}
\dfrac{v(x)Dv(x)}{\sqrt{\kappa_1^2(\kappa_1-\kappa_2)^2+|Dv(x)|^2}+\kappa_1(1-\kappa_1\kappa_2)}&=\bar Sx\label{eq:pdeforv},\\
\dfrac{C}{1-\kappa_1\kappa_2}(\kappa_1-\kappa_2)\sqrt{\kappa_1^2-1}<&v(x)<0 \label{eq:inequalityinv},
\end{align}
with $v(x_0)=\left(\delta_0+\dfrac{C}{1-\kappa_1\kappa_2}\right)(\kappa_1-\kappa_2)\sqrt{\kappa_1^2-1}:=z_0$.


Assume this $v$ exists. Taking absolute values in \eqref{eq:pdeforv} yields
\begin{equation}\label{eq:conditionSandv}
|\bar Sx|<|v(x)|=-v(x),
\end{equation}
and taking absolute values in \eqref{eq:pdeforv} again, and squaring both sides, we obtain that
\begin{equation}\label{eq:intermediaryresult}
v^2(x)\,|Dv(x)|^2=\left(\sqrt{\kappa_1^2(\kappa_1-\kappa_2)^2+|Dv(x)|^2}+\kappa_1(1-\kappa_1\kappa_2)\right)^2|\bar Sx|^2.
\end{equation}
Let $t(x):=\sqrt{\kappa_1^2(\kappa_1-\kappa_2)^2+|Dv(x)|^2}$, so
$|Dv(x)|^2=t^2(x)-\kappa_1^2(\kappa_1-\kappa_2)^2$. 
Replacing in \eqref{eq:intermediaryresult} yields
\[
v^2(x)(t^2(x)-\kappa_1^2(\kappa_1-\kappa_2)^2)=\left(t(x)+\kappa_1(1-\kappa_1\kappa_2)\right)^2|\bar Sx|^2.
\] 
Therefore $t(x)$ satisfies the quadratic equation
\begin{equation}\label{eq:quadraticfort}
(v^2(x)-|\bar Sx|^2)\,X^2-\left[2\kappa_1(1-\kappa_1\kappa_2)|\bar Sx|^2\right]\,X-\kappa_1^2\left[(\kappa_1-\kappa_2)^2v^2(x)+(1-\kappa_1\kappa_2)^2|\bar Sx|^2\right]=0.
\end{equation}
From \eqref{eq:conditionSandv} the constant coefficient is negative and  the quadratic coefficient is positive,
hence \eqref{eq:quadraticfort} has two solutions with opposite signs.
The discriminant $\Delta$ of \eqref{eq:quadraticfort} is
$$
\Delta=4\kappa_1^2v^2(x)\left[(\kappa_1-\kappa_2)^2v^2(x)-|\bar Sx|^2(1-\kappa_2^2)(\kappa_1^2-1)\right]>0.
$$
Since $t(x)>0$, we obtain
\begin{equation}\label{eq:formulafort}
t(x)=\bar t(x):=\dfrac{\kappa_1(1-\kappa_1\kappa_2)|\bar Sx|^2+\kappa_1|v(x)|\sqrt{(\kappa_1-\kappa_2)^2v^2(x)-|\bar Sx|^2(1-\kappa_2^2)(\kappa_1^2-1)}}{v^2(x)-|\bar Sx|^2},
\end{equation}
and therefore $t(x)$ can be written as a function of $v$ and $|\bar S|$.
Therefore equation \eqref{eq:pdeforv} can be written in the following way
\begin{equation}\label{eq:quasilinear}
Dv(x)=F(x,v(x)),
\end{equation}
with 
{\small
\begin{align}\label{eq:formulaforF}
F(x,v(x))&=\dfrac{\bar Sx}{v(x)}\left(\dfrac{\kappa_1(1-\kappa_1\kappa_2)|\bar Sx|^2+\kappa_1|v(x)|\sqrt{(\kappa_1-\kappa_2)^2v^2(x)-|\bar Sx|^2(1-\kappa_2^2)(\kappa_1^2-1)}}{v^2(x)-|\bar Sx|^2}+\kappa_1(1-\kappa_1\kappa_2)\right)\\
&=G\left(\dfrac{\bar Sx}{v(x)}\right),\notag
\end{align}
where 
\begin{equation}\label{eq:Formulaforg}
G(x)=\left(\dfrac{\kappa_1(1-\kappa_1\kappa_2)|x|^2+\kappa_1\sqrt{(\kappa_1-\kappa_2)^2-(1-\kappa_2^2)(\kappa_1^2-1)|x|^2}}{1-|x|^2}+\kappa_1(1-\kappa_1\kappa_2)\right)x:=\phi(|x|^2)x,
\end{equation}
}
for $|x|<1$.

Set 
$
F(x,z)=G\left(\dfrac{\bar Sx}{z}\right)=\phi\left(\dfrac{|\bar Sx|^2}{z^2}\right) \dfrac{\bar Sx}{z}=\left(G_1\left(\dfrac{\bar Sx}{z}\right),
G_2\left(\dfrac{\bar Sx}{z}\right)\right):=(F_1(x,z),F_2(x,z))$;
$\bar Sx=(\bar S_1x, \bar S_2x)$.

We have shown that if $v$ solves \eqref{eq:pdeforv} and \eqref{eq:inequalityinv}, then $v$ solves \eqref{eq:quasilinear}.

Conversely, we will solve \eqref{eq:quasilinear} and from this will obtain the solution to \eqref{eq:pdeforv} and \eqref{eq:inequalityinv}.
To this purpose we use a result from \cite[Chapter 6, pp. 117-118]{hartman-book-odes}. That is, if we assume that
\begin{equation}\label{eq:Hartmancondition}
\dfrac{\partial F_1}{\partial x_2}(x,z)+\dfrac{\partial F_1}{\partial z}(x,z)F_2(x,z)=\dfrac{\partial F_2}{\partial x_1}(x,z)+\dfrac{\partial F_2}{\partial z}(x,z)F_1(x,z)
\end{equation}
holds for all $(x,z)$ in a open set  $O$, then for each $(x_0,z_0)\in O$ there exists a neighborhood $V$ of the point $x_0$ and a unique solution $v$ to \eqref{eq:quasilinear} in $V$ satisfying $v(x_0)=z_0$. 
In our case, \eqref{eq:Hartmancondition} follows from the assumptions 
\eqref{firstcurlconditionforsolution} and
\eqref{eq:secondcurlconditionforsolution}. We postpone this verification for later (we will actually prove that \eqref{eq:Hartmancondition} is equivalent to \eqref{firstcurlconditionforsolution} and
\eqref{eq:secondcurlconditionforsolution}).

We then show that using 
Hartman's result, solving \eqref{eq:quasilinear} with an appropriate initial condition, furnishes the solution of \eqref{eq:pdeforv} and \eqref{eq:inequalityinv}.
In fact, as a consequence of \eqref{eq:conditionatx_0} we can pick $z_0<0$ satisfying
\begin{equation}\label{eq:choiceofz_0}
|\bar Sx_0|<|z_0|<\dfrac{-C}{1-\kappa_1\kappa_2}(\kappa_1-\kappa_2)\sqrt{\kappa_1^2-1}.
\end{equation}
Notice that \eqref{eq:choiceofz_0} holds if and only if \eqref{eq:initialconditiondelta0} holds for $\delta_0=\dfrac{z_0}{(\kappa_1-\kappa_2)\sqrt{\kappa_1^2-1}}-\dfrac{C}{1-\kappa_1\kappa_2}$.
Then with this choice of $z_0$, the solution $v(x)$ to \eqref{eq:quasilinear} with $v(x_0)=z_0$ satisfies \eqref{eq:inequalityinv} and \eqref{eq:conditionSandv} in a neighborhood $V'$ of $x_0$.
It remains to show that this solution $v$ satisfies the original equation \eqref{eq:pdeforv} in a neighborhood of $x_0$.
From the definition of $\bar t$ in \eqref{eq:formulafort} and \eqref{eq:formulaforF} we have
$F(x,v(x))=\dfrac{\bar Sx}{v(x)}\left(\bar t(x)+\kappa_1(1-\kappa_1\kappa_2) \right)$. Since $\bar t(x)$ solves \eqref{eq:quadraticfort}, we get
$v^2(x)(\bar t^2(x)-\kappa_1^2(\kappa_1-\kappa_2)^2)=\left(\bar t(x)+\kappa_1(1-\kappa_1\kappa_2)\right)^2|\bar Sx|^2$.
Therefore from \eqref{eq:quasilinear},
$|Dv(x)|^2=\bar t^2(x)-\kappa_1^2(\kappa_1-\kappa_2)^2$. 
Since $\bar t(x)>0$, we get that $\bar t(x)=t(x)$ and so equation \eqref{eq:quasilinear} is \eqref{eq:pdeforv} for $x\in V'$.

It remains to verify \eqref{eq:Hartmancondition}. From \eqref{eq:choiceofz_0} and by continuity, there is a neighborhood $O$ of the point $(x_0,z_0)$ such that $|\bar S x|<|z|$ for all $(x,z)\in O$ with $z<0$.
%
For all points $(x,z)\in O$ we have by calculation 
\begin{align*}
\dfrac{\partial F_1}{\partial x_2} (x,z)&=\dfrac{1}{z}DG_1\cdot \left(\dfrac{\partial \bar S_1}{\partial x_2}, \dfrac{\partial \bar S_2}{\partial x_2}\right)\qquad \dfrac{\partial F_2}{\partial x_1} (x,z)=\dfrac{1}{z}DG_2\cdot \left(\dfrac{\partial \bar S_1}{\partial x_1}, \dfrac{\partial \bar S_2}{\partial x_1}\right);\\
\dfrac{\partial F_1}{\partial z}(x,z)&=-\dfrac{1}{z^2}DG_1\cdot \bar S\qquad \qquad \hskip .5cm
\dfrac{\partial F_2}{\partial z}(x,z)=-\dfrac{1}{z^2}DG_2\cdot \bar S,
\end{align*}
where $DG_1$ and $DG_2$ are evaluated at $\dfrac{\bar Sx}{z}$; and $\bar S, \dfrac{\partial \bar S_i}{\partial x_j}$ are evaluated at $x$.
Moreover, from \eqref{eq:Formulaforg}, we have that
\begin{align*}
DG_1(x)&=\left(\phi\left(|x|^2\right)+2x_1^2\phi'\left(|x|^2\right), 2x_1x_2\phi'\left(|x|^2\right)\right),\\
DG_2(x)&=\left(2x_1x_2\phi'\left(|x|^2\right),\phi\left(|x|^2\right)+2x_2^2\phi'\left(|x|^2\right)\right).
\end{align*}
This implies that $(DG_1(x)\cdot x)G_2(x)=(DG_2(x)\cdot x)G_1(x)$,
which gives that $\dfrac{\partial F_1}{\partial z} F_2= \dfrac{\partial F_2}{\partial z} F_1$.
In addition, we obtain the identity for $(x,z)\in O$
{\small
\begin{align}\label{eq:simplifiedformula}
&\dfrac{\partial F_1}{\partial x_2}(x,z)-\dfrac{\partial F_2}{\partial x_1}(x,z)\notag
\\
&=\dfrac{1}{z}\phi\left(\dfrac{|\bar Sx|^2}{z^2}\right)\left(\dfrac{\partial \bar S_1}{\partial x_2}-\dfrac{\partial \bar S_2}{\partial x_1}\right)+
\dfrac{2}{z^3}\phi'\left(\dfrac{|\bar Sx|^2)}{z^2}\right)\left((\bar S_1x)^2\dfrac{\partial \bar S_1}{\partial x_2}- (\bar S_2x)^2\dfrac{\partial \bar S_2}{\partial x_1}+(\bar S_1x)(\bar S_2x)\left(\dfrac{\partial \bar S_2}{\partial x_2}-\dfrac{\partial \bar S_1}{\partial x_1}\right)\right)\\
&=
\dfrac{1}{z}\phi\left(\dfrac{|\bar Sx|^2}{z^2}\right)\,I(x)+
\dfrac{2}{z^3}\phi'\left(\dfrac{|\bar Sx|^2)}{z^2}\right)\,J(x),\notag
\end{align}
}
where 
\begin{align*}
I(x)&=\dfrac{\partial \bar S_1}{\partial x_2}-\dfrac{\partial \bar S_2}{\partial x_1}=-\text{curl } \bar S\\
J(x)&=(\bar S_1x)^2\dfrac{\partial \bar S_1}{\partial x_2}-(\bar S_2x)^2\dfrac{\partial \bar S_2}{\partial x_1}+(\bar S_1x)(\bar S_2x)\left(\dfrac{\partial \bar S_2}{\partial x_2}-\dfrac{\partial \bar S_1}{\partial x_1}\right)
=\dfrac12 \,\bar S\times D|\bar S|^2.
\end{align*}
Notice that \eqref{eq:Hartmancondition} holds if and only if \eqref{eq:simplifiedformula} equals zero.
Now from assumptions \eqref{firstcurlconditionforsolution} and \eqref{eq:secondcurlconditionforsolution}, and since $\bar S$ is a constant multiple of $S$, we then obtain \eqref{eq:Hartmancondition} for $x$ in a neighborhood of $x_0$.

Conversely, we shall prove that if $|\bar Sx_0|<|z_0|$ with $z_0<0$ and \eqref{eq:Hartmancondition} holds in a neighborhood $O$ of $(x_0,z_0)$ with $|\bar S x|<|z|$ in $O$ and $z<0$, then this implies that \eqref{firstcurlconditionforsolution} and \eqref{eq:secondcurlconditionforsolution} hold in a neighborhood of $x_0$.
From \eqref{eq:Formulaforg}
\begin{align*}
\phi(r)
&=\kappa_1\left(\dfrac{(1-\kappa_1\kappa_2)r+\Delta(r)}{1-r}+(1-\kappa_1\kappa_2)\right),
\end{align*}
$0\leq r<1$, where we have set $\Delta(r)=\sqrt{(\kappa_1-\kappa_2)^2-(1-\kappa_2^2)(\kappa_1^2-1)r}$.
Differentiating with respect to $r$ we obtain
$$
\phi'(r)=\dfrac{2\kappa_1(1-\kappa_1\kappa_2)\Delta(r)-\kappa_1(1-\kappa_2^2)(\kappa_1^2-1)(1+r)+2\kappa_1(\kappa_1-\kappa_2)^2 }{2(1-r)^2\Delta(r)}.
$$
Let $r=\dfrac{|\bar Sx|^2}{z^2}$ with $(x,z)\in O$. Since \eqref{eq:simplifiedformula} equals zero, we have
\begin{equation}\label{eq:multipliedequation}
z^2\phi\left(r\right)I(x)+2\phi'\left(r\right)J(x)=0.
\end{equation}
Will multiply this equation by $\dfrac{1}{\kappa_1}(1-r)^2 \Delta(r)$.
Notice first that 
\begin{align*}
\dfrac{1}{\kappa_1}(1-r)^2 \Delta(r)\phi\left(r\right)
&=
(1-\kappa_1\kappa_2)r(1-r)\Delta(r)+(1-r)\Delta(r)^2+(1-\kappa_1\kappa_2)(1-r)^2\Delta(r)\\
&=
(1-\kappa_1\kappa_2)(1-r)\Delta(r)+(1-r)\Delta(r)^2
\end{align*}
and
\[
\dfrac{1}{\kappa_1}(1-r)^2 \Delta(r)2\phi'\left(r\right)
=
2(1-\kappa_1\kappa_2)\Delta(r)-(1-\kappa_2^2)(\kappa_1^2-1)(1+r)+2(\kappa_1-\kappa_2)^2.
\]
Hence \eqref{eq:multipliedequation} can be written as
\begin{align*}
0&=z^2\,\left((1-\kappa_1\kappa_2)(1-r)\Delta(r)+(1-r)\Delta(r)^2\right)\,I(x)\\
&\qquad +
\left(2(1-\kappa_1\kappa_2)\Delta(r)-(1-\kappa_2^2)(\kappa_1^2-1)(1+r)+2(\kappa_1-\kappa_2)^2 \right)\,J(x)\\
&=
\left(z^2(1-\kappa_1\kappa_2)(1-r)\,I(x)+2(1-\kappa_1\kappa_2)\,J(x)\right)\,\Delta(r)\\
&\qquad +
z^2(1-r)\Delta(r)^2 I(x)+\left(-(1-\kappa_2^2)(\kappa_1^2-1)(1+r)+2(\kappa_1-\kappa_2)^2\right)J(x).
\end{align*}
This implies that
\begin{align*}
\left(z^2(1-\kappa_1\kappa_2)(1-r)\,I(x)+2(1-\kappa_1\kappa_2)\,J(x)\right)&\,\Delta(r)=\\
&z^2(r-1)\Delta(r)^2 I(x)+\left((1-\kappa_2^2)(\kappa_1^2-1)(1+r)-2(\kappa_1-\kappa_2)^2\right)J(x).
\end{align*}
Squaring both sides of this identity yields
\begin{align*}
\left(z^2(1-\kappa_1\kappa_2)(1-r)\,I(x)+2(1-\kappa_1\kappa_2)\,J(x)\right)^2\,&\Delta(r)^2
\\
&=\left( z^2(r-1)\Delta(r)^2 I(x)+\left((1-\kappa_2^2)(\kappa_1^2-1)(1+r)-2(\kappa_1-\kappa_2)^2\right)J(x)\right)^2,
\end{align*}
that is,
\begin{align*}
&\left(z^2(1-\kappa_1\kappa_2)(1-r)\,I(x)+2(1-\kappa_1\kappa_2)\,J(x)\right)^2\, \left((\kappa_1-\kappa_2)^2-(1-\kappa_2^2)(\kappa_1^2-1)r \right)\\
&=
\left( z^2(r-1)\left((\kappa_1-\kappa_2)^2-(1-\kappa_2^2)(\kappa_1^2-1)r \right) I(x)+\left((1-\kappa_2^2)(\kappa_1^2-1)(1+r)-2(\kappa_1-\kappa_2)^2\right)J(x)\right)^2.
\end{align*}
Substituting $r=|\bar Sx|^2/z^2$ yields
\begin{align*}
\left((1-\kappa_1\kappa_2)(z^2-|\bar Sx|^2)\,I(x)+2(1-\kappa_1\kappa_2)\,J(x)\right)^2&\,\left((\kappa_1-\kappa_2)^2\,z^2-(1-\kappa_2^2)(\kappa_1^2-1) \,|\bar Sx|^2\right)\dfrac{1}{z^2}=\\
&
\left[ (|\bar Sx|^2-z^2)\left((\kappa_1-\kappa_2)^2\,z^2-(1-\kappa_2^2)(\kappa_1^2-1) \,|\bar Sx|^2\right)\dfrac{1}{z^2} I(x)\right.\\
&\left.\qquad +\left((1-\kappa_2^2)(\kappa_1^2-1)\left(z^2+|\bar Sx|^2\right)-2(\kappa_1-\kappa_2)^2\,z^2\right)\dfrac{1}{z^2}J(x)\right]^2.
\end{align*}
Multiplying by $z^4$ yields
\begin{align}\label{eq:polynomial1}
z^2\,\left((1-\kappa_1\kappa_2)(z^2-|\bar Sx|^2)\,I(x)+2(1-\kappa_1\kappa_2)\,J(x)\right)^2&\,\left((\kappa_1-\kappa_2)^2\,z^2-(1-\kappa_2^2)(\kappa_1^2-1) \,|\bar Sx|^2\right)=\\
&
\left[ (|\bar Sx|^2-z^2)\left((\kappa_1-\kappa_2)^2\,z^2-(1-\kappa_2^2)(\kappa_1^2-1) \,|\bar Sx|^2\right) I(x)\right.\notag\\
&\left.\qquad \qquad+\left((1-\kappa_2^2)(\kappa_1^2-1)\left(z^2+|\bar Sx|^2\right)-2(\kappa_1-\kappa_2)^2\,z^2\right)J(x)\right]^2.\notag
\end{align}
Both sides of the last identity are polynomials of 8th degree in $z$ with coefficients that depend on $x$. Comparing the coefficients of $z^8$ in \eqref{eq:polynomial1}, we get that
\[
(1-\kappa_1\kappa_2)^2(\kappa_1-\kappa_2)^2\, I^2(x)=(\kappa_1-\kappa_2)^4\,I^2(x).
\]
Since $(1-\kappa_1\kappa_2)^2-(\kappa_1-\kappa_2)^2=(1-\kappa_1^2)(1-\kappa_2^2)\neq 0$, it follows that $I(x)=0$ in a neighborhood of $x_0$
Thus replacing $I(x)=0$ in \eqref{eq:polynomial1} yields
\begin{align}\label{eq:polynomial2}
\left(2\,z\, (1-\kappa_1\kappa_2)\,J(x)\right)^2&\left(\left(\kappa_1-\kappa_2\right)^2\,z^2-(1-\kappa_2^2)(\kappa_1^2-1)|\bar Sx|^2\right)=\\
&\left[\left((1-\kappa_2^2)(\kappa_1^2-1)\left(z^2+|\bar S x|^2\right)-2\left(\kappa_1-\kappa_2\right)^2z^2\right)J(x)\right]^2. \notag
\end{align}
Now both sides of \eqref{eq:polynomial2} are polynomials of 4th degree in $z$, and comparing the coefficients of $z^4$, we get that
\begin{equation}\label{eq:identitypolynomial2}
\left[2(1-\kappa_1\kappa_2)(\kappa_1-\kappa_2)\right]^2 J^2(x) =\left[(1-\kappa_2^2)(\kappa_1^2-1)-2(\kappa_1-\kappa_2)^2\right]^2 J^2(x).
\end{equation}
The coefficients of $J^2(x)$ in \eqref{eq:identitypolynomial2} are different since 
\[
\left[2(1-\kappa_1\kappa_2)(\kappa_1-\kappa_2)\right]^2-\left[(1-\kappa_2^2)(\kappa_1^2-1)-2(\kappa_1-\kappa_2)^2\right]^2=- (1-\kappa_1^2)^2\,(1-\kappa_2^2)^2\neq 0.
\]
Therefore $J(x)=0$ in a neighborhood of $x_0$.
Since $\bar S$ is a constant multiple of $S$, we obtain \eqref{firstcurlconditionforsolution} and \eqref{eq:secondcurlconditionforsolution}. This completes the proof of the theorem.

\end{proof}

\subsection{Remarks on conditions \eqref{firstcurlconditionforsolution} and \eqref{eq:secondcurlconditionforsolution}}\label{rmk:condtionsonS}

Condition \eqref{firstcurlconditionforsolution} is equivalent to the existence of function $w$ such that $S=(w_{x_1},w_{x_2})$. Therefore \eqref{eq:secondcurlconditionforsolution} is equivalent to the following quasi linear equation
\begin{equation}\label{eq:quasilinearpde}
w_{x_1x_2}\left((w_{x_1})^2-(w_{x_2})^2\right)+w_{x_1}\,w_{x_2}
\left(w_{x_2x_2}-w_{x_1x_1}\right)=0.
\end{equation}
Using Cauchy-Kowalevski's theorem, the Cauchy problem problem for this equation can be solved for a large class of initial data, \cite[Chapter 2]{pdebood:fritzjohn}.
In fact, if we are given two analytic curves $\gamma(s)=(f(s),g(s))$ and $\Gamma(s)=(\phi(s),\psi(s))$ satisfying the non characteristic condition
\[
\det
\left[\begin{matrix}
f' & g' & 0\\
0 & f' & g'\\
-\phi\,\psi & \phi^2-\psi^2 & \phi\,\psi
\end{matrix}
\right]\neq 0,
\]  
and $z$ satisfies the compatibility condition $z'=\phi\,f'+\psi\,g'$, then there exists a unique solution $w$ solving \eqref{eq:quasilinearpde} locally and satisfying the initial conditions 
$w(f(s),g(s))=z(s)$, $w_{x_1}(f(s),g(s))=\phi(s)$ and $w_{x_2}(f(s),g(s))=\psi(s)$.
In particular, we can construct transformations $Tx$ such that $Sx=Tx-x$ satisfy \eqref{firstcurlconditionforsolution} and \eqref{eq:secondcurlconditionforsolution} and map the curve $\gamma$ into the curve $\Gamma$.

Notice that $S$ verifies condition \eqref{eq:conditionsamemedium}, when $S$ satisfies \eqref{firstcurlconditionforsolution} and \eqref{eq:secondcurlconditionforsolution}. 

\subsection{Solution with a different method in dimension two for a magnification}\label{sec:legendretransform} 
We present in this section a different method to find $u$ when $T$ is $Tx=(1+\alpha)x$, $\alpha\neq 0$.
When the dimension is two, 
equation \eqref{eq:secondformpde} becomes
\begin{equation}\label{eq:pdeugeneral}
\dfrac{(1-\kappa_1\kappa_2)u(x)+C}{(\kappa_1^2-\kappa_1\kappa_2)\sqrt{\kappa_1^2+(\kappa_1^2-1)|u'(x)|^2}+\kappa_1^2(1-\kappa_1\kappa_2)}u'(x)=\dfrac{\alpha x}{\kappa_1^2-1}.
\end{equation}
We use the Legendre transform to change variables in \eqref{eq:pdeugeneral}; we set
$$w(\xi)=-u(x)+x\xi,\,\,\,  u'(x)=\xi, \,\,\,  w'(\xi)=x.$$
Making this change of variables in \eqref{eq:pdeugeneral} yields the ode
$$\left[(1-\kappa_1\kappa_2)(w'(\xi)\xi-w(\xi))+C\right]\xi=\dfrac{\alpha}{\kappa_1^2-1} \left[(\kappa_1^2-\kappa_1\kappa_2)\sqrt{\kappa_1^2+(\kappa_1^2-1)\xi^2}+\kappa_1^2(1-\kappa_1\kappa_2)\right]w'(\xi),
$$
which setting $r(\xi)=(1-\kappa_1\kappa_2)\xi^2-\dfrac{\alpha}{\kappa_1^2-1}\left[(\kappa_1^2-\kappa_1\kappa_2)\sqrt{\kappa_1^2+(\kappa_1^2-1)\xi^2}+\kappa_1^2(1-\kappa_1\kappa_2)\right]$ can be written as
\[
r(\xi)w'(\xi)+(\kappa_1\kappa_2-1)\xi\,w(\xi)=-C\,\xi.
\]
If $\alpha<0$, then $r(\xi)>0$ for all $\xi$. On the other hand, if $\alpha>0$, then there is $\xi_1> 0$ such that $r(\pm \xi_1)=0$.
That is, $w$ satisfies a linear ode of form $w'(\xi)+P(\xi)w(\xi)=Q(\xi)$ with
\begin{equation}\label{eq:formulaforP}
P(\xi)=- \dfrac{(1-\kappa_1\kappa_2)\xi}{r(\xi)},\qquad 
Q(\xi)=-\dfrac{C\xi}{r(\xi)},
\end{equation}
on the intervals $(-\infty,-\xi_1), (-\xi_1,\xi_1),(\xi_1,+\infty)$.
Note that $Q(\xi)=\dfrac{C\,P(\xi)}{1-\kappa_1\kappa_2}$ then 
\begin{align*}
0=w'(\xi)+P(\xi)w(\xi)-\dfrac{C}{1-\kappa_1\kappa_2}P(\xi)
=\left(w-\dfrac{C}{1-\kappa_1\kappa_2}\right)'+P(\xi)\left(w-\dfrac{C}{1-\kappa_1\kappa_2}\right).
\end{align*}
Hence 
\begin{equation}\label{eq:legendresolution}
w(\xi)-\dfrac{C}{1-\kappa_1\kappa_2}=\dfrac{A}{e^{\int P(\xi)d\xi}}.
\end{equation}
Differentiating $w$ twice we obtain 
\begin{equation}\label{eq:formulaforw''bis}
w''(\xi)=Ae^{-\int P(\xi)d\xi} (P(\xi)^2-P'(\xi)).\\
\end{equation}
We have
$
r'(\xi)=2(1-\kappa_1\kappa_2)\xi -\dfrac{\alpha(\kappa_1^2-\kappa_1\kappa_2)\xi}{\sqrt{\kappa_1^2+(\kappa_1^2-1)\xi^2}},
$
and then from \eqref{eq:formulaforP}
\begin{align*}
P(\xi)^2-P'(\xi)&=\dfrac{(1-\kappa_1\kappa_2)^2\xi^2}{r^2(\xi)}+\dfrac{(1-\kappa_1\kappa_2)(r(\xi)-\xi r'(\xi))}{r^2(\xi)}\\
&=\dfrac{-\alpha(1-\kappa_1\kappa_2)\kappa_1^2}{r^2(\xi)(\kappa_1^2-1)\sqrt{\kappa_1^2+(\kappa_1^2-1)\xi^2}}\left[\kappa_1^2-\kappa_1\kappa_2+(1-\kappa_1\kappa_2)\sqrt{\kappa_1^2+(\kappa_1^2-1)\xi^2}\right].
\end{align*}
The term between brackets is positive since $\kappa_1\kappa_2<1.$
Since $\alpha\neq 0$, then the function $P^2(\xi)-P'(\xi)$ is never zero for $\xi\neq \pm \xi_1$ and has the sign of $\alpha (\kappa_1\kappa_2-1)$. Therefore from \eqref{eq:formulaforw''bis}, $w''$ is never zero for $\xi\neq \pm \xi_1$ and has the sign of $A\,\alpha\,(\kappa_1\kappa_2-1)$. 
We then obtain that $w'$ is strictly monotone and hence it can be inverted, i.e., $\xi=(w')^{-1}(x)$. We conclude that the solution $u$ to \eqref{eq:pdeugeneral} is given by $u(x)=x\,(w')^{-1}(x)-w((w')^{-1}(x))$.

\setcounter{equation}{0}
\section{Solution of the Imaging Problem When $n_1<n_3$}\label{sec:n3larger}
By reversibility of the optical paths this case reduces to the case when $\kappa_1\kappa_2<1$ from Section \ref{sec:n3smaller} by switching the roles of $n_3$ and $n_1$. 
Notice that $\kappa_1\kappa_2>1$ if and only if $n_3>n_1$.
Let $T:\Omega\to \Omega^*$ be a bijection, then $T^{-1}:\Omega^*\to \Omega$.
By reversibility of the optical paths a lens sandwiched between the surfaces $\sigma_1$ and $\sigma_2$ that images $\Omega$ into $T(\Omega)=\Omega^*$ exists if and only if the same lens images $\Omega^*$ into $T^{-1}(\Omega^*)=\Omega$. In the second case this means that the refractive index next to $\Omega^*$ is $n_3$ which is  
bigger that $n_1$ the refractive index next to $\Omega$, this means we are in the case from Section \ref{sec:n3smaller} with $n_3$ and $n_1$ switched. Therefore from Theorem \ref{eq:existensolutionpde}
a sufficient condition for the  local existence of the lens is that the components of the map $T^{-1}$ verify \eqref{firstcurlconditionforsolution} and \eqref{eq:secondcurlconditionforsolution}.

\setcounter{equation}{0}
\section{Solution of the far field and the near field problems for reflection}\label{sec:reflectioncase}
We briefly indicate how to extend the results obtained for reflectors.
\subsection{Snell's law of reflection}
If a ray with unit direction $x$ strikes a reflective surface $\Gamma$ at a point $P$, then  this ray is reflected with unit direction $m$ such that 
$$x\times \nu =m\times \nu,$$ where $\nu$ is a unit normal at $P$. We obtain then that
\begin{equation}\label{eq:snelllawreflection}
x-m=2(x\cdot \nu)\nu.
\end{equation}

\subsection{Far Field Problem for reflection}
Consider the problem analyzed in Section \ref{sec:farfield}, now for reflection, with $\sigma_1$ a given $C^2$ reflective surface given by the graph of a function $u$. Rays emanate from $(x,0)$ with $x\in \Omega$ and with unit direction $e(x)$ as in Section \ref{sec:farfield}. Then proceeding similarly as in Section \ref{sec:farfield}, we obtain the analogue of Theorem \ref{thm:existenceofupperlenssurfaces} for reflection.
That is, we obtain that a surface $\sigma_2$ exists, such that rays reflected by $\sigma_1$, with direction $m_1(x)$, are then reflected by $\sigma_2$ into a fixed unit direction $w$, if and only if $e'(x)=\nabla h(x)$ for some $C^2$ function $h$, and $m_1(x)\cdot w<1$. In this case, $\sigma_2$ is $C^2$ and is parametrized by the vector
$f(x)=\left(\varphi(x),u(\varphi(x))\right)+d(x)m_1(x)$
with
\begin{align}
m_1(x)&=e(x)-2(e(x)\cdot \nu_1)\nu_1,\label{eq:definitionofm1forreflection}\\
d(x)&=\dfrac{C-h(x)+e(x)\cdot (x,0)-(e(x)-w)\cdot (\varphi(x),u(\varphi(x)))}{1-w\cdot m_1(x)},\label{eq:definitionofdforreflection}
\end{align}
where $\nu_1$ is the unit normal to $\sigma_1$ at $(\varphi(x),u(\varphi(x))$.

Moreover, as shown in Example \ref{sec:Pointsource}, using the above analysis one can deduce the far field case considered in \cite[Section 3]{gutierrez-sabra:designofpairsofreflectors} where rays are assumed to be emitted from a point source.

\subsection{Near Field Imaging Problem}
We obtain the following analogue of Theorem \ref{thm:resultwhensameindex} for reflection.
\begin{theorem}\label{thm:nearfieldreflector}
Let $T$ be a $C^1$ bijective map between the set $\Omega$ and a plane domain $\Omega^*$, and $a>0$. Then there exists a system of two mirrors $(\sigma_1,\sigma_2)$ reflecting all rays emitted from $(x,0), \, x\in \Omega$, with direction $e_3=(0,0,1)$ into the point $(Tx,a)$ such that rays leave the lens with direction $e_3$ if and only if
$$\text{\rm curl } S=0,$$
where $Sx=Tx-x$.
In this case, the lower face of the lens $\sigma_1$ is parametrized by $(x,u(x))$ where $u$ is given by \eqref{eq:solutionreflection}, and the top face $\sigma_2$ of the lens is parametrized by the vector $f(x)=(x,u(x))+d(x)m_1(x)$, with $d$ and $m_1$ given by \eqref{eq:formulafordreflection} and \eqref{eq:m1forreflection}, respectively. Both surfaces $\sigma_1$ and $\sigma_2$ are $C^2$.
\end{theorem}

\begin{proof}
In this case from \eqref{eq:definitionofdforreflection} and \eqref{eq:definitionofm1forreflection}
\begin{equation}\label{eq:formulafordreflection}
d(x)= \dfrac{C}{1-m_1(x)\cdot e_3},
\end{equation}
and 
\begin{equation}\label{eq:m1forreflection}
m_1(x)=(0,0,1)-2\left((0,0,1)\cdot \dfrac{(-Du(x),1)}{\sqrt{1+|Du(x)|^2}}\right) \dfrac{(-Du(x),1)}{\sqrt{1+|Du(x)|^2}}
=\left(\dfrac{2 Du(x)}{1+|Du(x)|^2},1-\dfrac{2}{1+|Du(x)|^2}\right).
\end{equation}

As in Section \ref{sec:nearfield}, we have $Tx= (f_1(x),f_2(x))$, where $f(x)=\left(x,u(x)\right)+d(x)\,m_1(x)$, and hence
\begin{align*}
Tx
&=x+ C\, Du(x).
\end{align*}
Therefore $u$ satisfies the pde
\begin{equation}\label{eq:pdereflection}
Du(x)=\dfrac{Sx}{C},
\end{equation}
where $Sx=Tx-x$.
If $x,x_0\in \Omega$, with $\Omega$ simply connected, then one obtains
\begin{equation}\label{eq:solutionreflection}
u(x)=u(x_0)+\dfrac{1}{C}\int_{\mathcal C} Sx\cdot dr
\end{equation}
 if and only if $\text{curl} (Sx)=0$, where the integral is a line integral along an arbitrary path $\mathcal C$ from $x_0$ to $x$.
 \end{proof}

\section{Conclusion}
We have shown an explicit method to design a lens that refracts a given variable unit field of directions $e(x)$ into a prescribed direction $w$. 
For the lens to exist, we show that the field must satisfy a curl-zero condition and a compatibility condition with $w$.
Using this analysis we solve an imaging problem consisting in finding a lens that focuses two plane images $A$ and $B$, one into another. We find sufficient conditions on the map $T:A\to B$, $B=T(A)$, that guarantee the existence of the lens. These conditions depend on the relationships between the refractive indices. Similar problems for reflectors are solved.
The methods used consist in solving first order systems of pdes, and we illustrate the results with several examples.

%

\end{document}